\documentclass{article}

\usepackage[margin=2cm]{geometry}
\usepackage{color}
\usepackage{subcaption}
\usepackage{graphicx}
\usepackage{algorithm,algpseudocode}
\usepackage{authblk}
\usepackage{amsmath,amsfonts,amsthm,amssymb}
\usepackage{enumerate}

\newcommand{\R}{\mathbb{R}}
\newcommand{\N}{\mathbb{N}}
\newcommand{\calh}{\mathcal H}
\newcommand{\fa}{\hbox{ for all }}

\DeclareMathOperator{\diag}{diag}

\DeclareMathOperator{\Span}{span}

\newcommand{\pol}{\mathbb{P}}

\newcommand{\inner}[2]{\ifthenelse{\equal{#2}{}}{\left\langle\cdot,\cdot\right\rangle_{#1}}{\left\langle#2\right\rangle_{#1}}}
\newcommand{\Norm}[2]{\ifthenelse{\equal{#2}{}}{\left\|\cdot\right\|_{#1}}{\left\|#2\right\|_{#1}}}
\newcommand{\seminorm}[2]{\ifthenelse{\equal{#2}{}}{\left|\cdot\right|_{#1}}{\left|#2\right|_{#1}}}
\newcommand{\hpol}{\mathbb{H}}

\newtheorem{theorem}{Theorem}
\newtheorem{proposition}[theorem]{Proposition}
\newtheorem{corollary}[theorem]{Corollary}
\newtheorem{lemma}[theorem]{Lemma}
\newtheorem{definition}[theorem]{Definition} 
\newtheorem{remark}[theorem]{Remark}

\title{Interpolation with the polynomial kernels}

\author[1]{Giacomo Elefante}
\author[1]{Wolfgang Erb}
\author[1]{Francesco Marchetti} 
\affil[1]{Dipartimento di Matematica \lq\lq Tullio Levi-Civita\rq\rq, Universit{\`a} degli Studi di Padova (Italy)}

\author[2]{Emma Perracchione}
\affil[2]{Dipartimento di Scienze Matematiche \lq\lq Giuseppe Luigi Lagrange\rq\rq, Politecnico di Torino (Italy)}

\author[3]{Davide Poggiali}
\affil[3]{FAR Networks S.r.l., Cernusco S.N., Milano (Italy)}

\author[4]{Gabriele Santin}
\affil[4]{Digital Society Center (DIGIS), Fondazione Bruno Kessler, Trento (Italy)}

\begin{document}
\maketitle

\begin{abstract}
The polynomial kernels are widely used in machine learning and they are one of the default choices to develop kernel-based classification and regression models.
However, they are rarely used and considered in numerical analysis due to their lack of strict positive definiteness. In particular they do not enjoy the 
usual property of unisolvency for arbitrary point sets, which is one of the key properties used to build kernel-based interpolation methods.

This paper is devoted to establish some initial results for the study of these kernels, and their related interpolation algorithms, in the context of 
approximation theory. 
We will first prove necessary and sufficient conditions on point sets which guarantee the existence and uniqueness of an interpolant. We will 
then study the Reproducing Kernel Hilbert Spaces (or native spaces) of these kernels and their norms, and provide inclusion relations between spaces 
corresponding to different kernel parameters.  With these spaces at hand, it will be further possible to derive generic error estimates which apply to 
sufficiently smooth functions, thus escaping the native space. Finally, we will show how to employ an efficient stable algorithm to these kernels to obtain 
accurate interpolants, and we will test them in some numerical experiment.
After this analysis several computational and theoretical aspects remain open, and we will outline possible further research directions in a concluding section.

This work builds some bridges between kernel and polynomial interpolation, two topics to which the authors, to different extents, have been introduced under 
the supervision or through the work of Stefano De Marchi. For this reason, they wish to dedicate this work to him in the occasion of his 60th birthday. 
\end{abstract}

\section{Introduction}\label{sec:intro}
Positive definite kernels are widely used in a variety of problems ranging from numerical analysis to machine learning, including Gaussian process 
regression. 

In different fields they come into play from different directions. In numerical analysis, they provided data-dependent bases that permit interpolation of 
scattered data \cite{Buhmann2003,Fasshauer2007,Wendland2005}; in machine learning, they are usually the result of the application of a feature map on the 
input data, used to transform linear algorithms into nonlinear ones by means of an high dimensional space \cite{Schoelkopf2002,Shawe-Taylor2004,Steinwart2008}; 
in Gaussian process regression, they represent the covariance function of a stochastic process \cite{Rasmussen2006}.

Despite this remarkable variety, a large part of the success of kernel methods is due to the fact that they can be analyzed to some extent within the unified 
framework of Reproducing Kernel Hilbert Spaces (RKHS) \cite{Saitoh2016}, which provide a solid mathematical underpinning to different algorithmic approaches. This 
connection has the additional benefit that novel ideas and points of view may often spread from one field to another through this common perspective, see e.g. \cite{Belkin2018,Karvonen2020a,McCourt2017,Pagliana2020,Scheuerer2013} for a few recent examples. 
% Notable cases are e.g. 
% \cite{Belkin2018}
% (Karvonen, Fasshauer, Rosasco, Belkin, ), just to name a few recent examples.

However, it is still the case that some requirements and conditions are peculiar to one or the other specific setting, and thus this spill over is not always 
possible. In particular, a major difference between the point of view of machine learning and Gaussian process on one hand, and the one of numerical analysis 
on the other, is the definiteness of the kernel in the following sense.
\begin{definition}[Definiteness classes]\label{def:pdclasses} 
Let $\Omega\neq \emptyset$ be a set and let $k:\Omega\times \Omega\to \R$ be symmetric. Then $k$ is said to be positive definite on $\Omega$ if for all 
$N\in\N$ and for all sets 
$X:=\{x_i\}_{i=1}^N\subset \Omega$ the kernel matrix $A:=(k(x_i, x_j))_{i,j=1}^N\in\R^{N\times N}$ is positive semidefinite. The kernel is additionally said 
to be strictly positive definite if $A$ is 
positive definite whenever the points in $X$ are pairwise distinct. 
\end{definition}
In the analysis of stochastic processes kernels are used as covariance functions, which are in general only positive definite. 
In machine learning, data approximation models are usually defined as the solution of an optimization problem, which can be proven to be convex even if the 
employed 
kernel is just positive definite. The approximation models considered in numerical analysis are instead mostly interpolatory, and their existence is guaranteed 
for general distributions of the interpolation points only if the kernel is strictly positive definite. Namely, given an input space $\Omega\neq\emptyset$, a 
set $X:=\{x_i\}_{i=1}^N\subset \Omega$ of pairwise distinct interpolation points, and a set $Y:=\{y_i\}_{i=1}^N\subset \R$ of target values, a kernel is used 
to build an interpolatory model
\begin{equation}\label{eq:kernel_model}
s(x):=\sum_{i=1}^N c_i k(x, x_i),\;\; x\in\Omega,\quad\quad s(x_i) = y_i,  \fa 1\leq i\leq N.
\end{equation}
This model exists precisely when there exists a vector $c:=(c_1, \dots, c_N)^T\in\R^N$ which 
solves 
\begin{equation}\label{eq:kernel_system}
A c = y,
\end{equation}
with $y:=(y_1, \dots, y_N)^T\in\R^N$ and $A$ as in Definition \ref{def:pdclasses}. The solvability of this system is in turn guaranteed if $A$ is positive 
definite and thus invertible for any set $X$, i.e., if $k$ is strictly positive definite. In other words, strictly positive definite kernels are used in 
numerical analysis to construct data-dependent interpolation models of the form \eqref{eq:kernel_model}, thus enabling the interpolation of arbitrarily 
scattered data for arbitrary input space dimension. This capability in particular permits to overcome the limitations of classical techniques such as 
polynomial interpolation, which require instead precise geometrical conditions on the interpolation points. 

This distinction has the effect that many kernels which are commonly employed in machine learning and Gaussian process regression 
are almost unknown in approximation theory and numerical analysis, since they are only positive definite. In particular, in this paper we consider the 
notable family of polynomial kernels $k_{a, p}(x, y):= (a + \inner{}{x, y})^p$ defined for $a\geq0$ and $p\in\N$ on a subset $\Omega\neq \emptyset$ of the 
Euclidean space $\R^d$. These kernels are widely used in machine learning, where they are often even considered the essential basic example of a positive 
definite kernel. For example, together with the Gaussian kernel, polynomial kernels are the only ones implemented by default in the widespread 
Scikit-Learn Python machine learning library \cite{sklearn_api,scikit-learn}, as well as in the Matlab Statistics and Machine Learning Toolbox \cite{Matlab}.

Due to their lack of strict positive definiteness, they have however received little attention in approximation theory. For this reason in this work we aim at 
establishing an interpolation theory for these polynomial kernels. In particular, after recalling some additional details on their definition and properties in 
Section \ref{sec:background}, we characterize in Section \ref{sec:unisolvent} sets of points $X$ which are unisolvent for the polynomial kernels, i.e., which 
allow unique interpolation. As proven in Theorem \ref{thm:existence}, we obtain the remarkable result that any set of pairwise distinct points in $\R^d$ is 
unisolvent provided $p$ is chosen large enough. After the existence of an interpolant is established, we obtain in Section \ref{sec:native_and_error} an error 
bound for its approximation error. This in particular allows one to ``escape the native space'' in the sense of \cite{Narcowich2006}, i.e., approximating 
functions which are outside of the RKHS of the kernel. In the same section we also use an argument of \cite{Zwicknagl2009} to provide a characterization of the 
RKHS, which turns out to be simply the space $\pol_p^d(\Omega)$ of $d$ variate polynomials of degree $p$ over $\Omega$ or the corresponding homogeneous space 
$\hpol_p^d(\Omega)$, each equipped with a suitable inner product. Moreover, we study the stability of this interpolation process in Section \ref{sec:stability}, 
and show that, although the direct solution of the linear system \eqref{eq:kernel_system} is possibly highly unstable, one can apply the celebrated RBF-QR algorithm \cite{Fasshauer2012b,Fornberg2011} to $k_{a,p}$, thus obtaining stable computations.
From this analysis of existence, convergence, and stability it turns out that polynomial-kernel interpolation is strictly related to standard polynomial 
interpolation, perhaps unsurprisingly. In particular, we argue that good interpolation points for these kernels can be found from good interpolation points for 
polynomial interpolation, such as \cite{Bos2006,Bos2010, Bos2011b,Bos2007,Caliari2005,Caliari2008}.
Finally, we test our findings in a number of experiments in Section \ref{sec:numerics}, and comment on some possible extensions in Section \ref{sec:conclusions}.

\section{Background on polynomials and the polynomial kernels}\label{sec:background}

\subsection{Multivariate polynomial spaces}
We start by recalling some notation and the necessary background results on multivariate polynomials.

Let $d\in\N$ and $p\in\N$. Given a multiindex $\zeta:=(\zeta_1, \dots, \zeta_d)\in\N_0^d$, we write $|\zeta|:=\zeta_1+\ldots+\zeta_d$ for its length and
$\zeta!:=\zeta_1!\cdot\ldots\cdot\zeta_d!$ for its factorial, and denote the monomial with degrees $\zeta$ as $x^\zeta:=\prod_{i=1}^d x_i^{\zeta_i}$, 
$x\in\R^d$. 
For two multiindices $\zeta,\beta\in\N_0^d$, the term $\delta_{\zeta\beta}$ has value $1$ when $\zeta_i=\beta_i$ for all $1\leq i\leq d$, and otherwise it has value zero.

We denote as $\pol_p^d$ the space of polynomials over $\R^d$ of total degree at most $p$, and as $\hpol_p^d$ the corresponding homogeneous space, 
i.e., the space of polynomials over $\R^d$ of total degree exactly $p$. The two spaces have dimension $M_{p}^d:= \binom{d+p}{d}=\dim(\pol_p^d)$ and 
$M_{p}^{d-1} = \dim(\hpol_p^d)$, respectively.

Using a notation that will be motivated in the next section, for any $a\geq 0$ we denote two sets of multi-indices that we will use repeatedly as
\begin{equation}\label{eq:index_sets}
I_a(p, d):=
\begin{cases}
\left\{\zeta \in \N_0^{d}, |\zeta|\leq p\right\}, & a>0\\
\left\{\zeta \in \N_0^{d}, |\zeta|= p\right\}, & a=0,
\end{cases}
\end{equation}
and set
\begin{equation}\label{eq:index_sets_dim}
M_a:=
\begin{cases}
M_{p}^d=\dim(\pol_{p}^d), & a>0\\
M_{p}^{d-1}=\dim(\hpol_{p}^d), & a=0,
\end{cases}
\end{equation}
so that $M_a= \dim(I_a(p, d))$.

\subsection{The polynomial kernels}
With these notations in hand, we can now give a formal definition of the family of polynomial kernels and state some of their properties. The content of 
this section is a collection 
of classical results, for which we refer e.g. to \cite{Schoelkopf2002,Shawe-Taylor2004,Steinwart2008}.

Let $d\in\N$, $\Omega\subset \R^d$, and let $a\geq0, p\in\N$ be fixed. The polynomial kernel $k:=k_{a, p}:\Omega\times\Omega\to\R$ is defined as $k(x, 
y):= \left(a + \left\langle x, y\right\rangle\right)^p$, $x, y\in\Omega$, where $\left\langle x,y\right\rangle$ is the inner product on $\R^d$.

If $a>0$, using the notation \eqref{eq:index_sets} and the multinomial expansion we have
\begin{equation}\label{eq:expansion_kernel}
k_{a,p}(x, y)
=\left(a + \sum_{i=1}^d x_i y_i \right)^p
=\sum_{\stackrel{(\zeta_0, \zeta) \in \N_0^{d+1}}{\zeta_0+|\zeta|= p}} \frac{p! a^{\zeta_0}}{\zeta_0!\zeta!}  x^\zeta y^\zeta
=\sum_{\zeta \in \N_0^{d}, |\zeta|\leq p} \frac{p! a^{p-|\zeta|}}{(p-|\zeta|)!\zeta!}  x^\zeta y^\zeta
=\sum_{\zeta \in I_a(p,d)} d_\zeta^a  x^\zeta y^\zeta,
\end{equation}
where for $a>0$ we defined 
\begin{equation}\label{eq:d_alpha}
d_\zeta^a:= \frac{p! a^{p-|\zeta|}}{(p-|\zeta|)!\zeta!}, \;\; \zeta \in I_a(p,d).
\end{equation}

Considering
an arbitrary enumeration $\{\zeta^{(1)}, \dots, \zeta^{(M_a)}\}$ of the set $I_a(p,d)$, with $M_a$ as defined in \eqref{eq:index_sets_dim},
the representation \eqref{eq:expansion_kernel} shows that the function $\Phi_{a,p}:\R^d\to\R^{M_a}$ defined for $a>0$ by
\begin{equation}\label{eq:feature_map}
\Phi_{a,p}(x) := \left[\sqrt{d_{\zeta^{(1)}}^a}\ x^{\zeta^{(1)}}, \sqrt{d_{\zeta^{(2)}}^a}\ x^{\zeta^{(2)}}, \dots, \sqrt{d_{\zeta^{(M_a)}}^a}\ 
x^{\zeta^{(M_a)}}\right]^T, \;\;x \in \R^d,
\end{equation}
is a feature map for $k_{a,p}$ on $\R^d$ with feature space $\R^{M_a}$, i.e., it holds $k_{a,p}(x, y) = \Phi_{a,p}(x)^T \Phi_{a,p}(y)$ for all $x,y\in\R^d$. 

If instead $a=0$, again with the notation \eqref{eq:index_sets} the same chain of reasoning as above proves that 
\begin{equation}\label{eq:expansion_kernel_hom}
k_{0,p}(x, y)
=\sum_{\zeta \in I_0(p,d)} d^0_\zeta  x^\zeta y^\zeta, \quad\quad d^0_\zeta:= \frac{p!}{\zeta!}, \;\; \zeta \in I_0(p,d).
\end{equation}
In this case, taking an arbitrary enumeration $\{\zeta^{(1)}, \dots, 
\zeta^{(M_0)}\}$ of this set, with $M_0$ as in \eqref{eq:index_sets_dim}, we obtain as before that a feature map for $k_{0, p}$ is $\Phi_{0,p}:\R^d\to\R^{M_0}$ with
\begin{equation}\label{eq:feature_map_hom}
\Phi_{0,p}(x) := \left[\sqrt{d^0_{\zeta^{(1)}}}\ x^{\zeta^{(1)}}, \sqrt{d^0_{\zeta^{(2)}}}\ x^{\zeta^{(2)}}, \dots, \sqrt{d^0_{\zeta^{(M_0)}}}\ 
x^{\zeta^{(M_0)}}\right]^T, \;\;x \in \Omega.
\end{equation}

Both for $a>0$ and $a=0$, the existence of a feature map implies that the polynomial kernel is positive definite, i.e., for each $X:=\{x_i\}_{i=1}^N\subset 
\Omega$ the kernel matrix $A:=(k_{a,p}(x_i, x_j))_{i,j=1}^N \in\R^{N\times N}$ is positive semidefinite. This is immediate from Definition \ref{def:pdclasses} 
since $A$ is the Gramian matrix of the $N$ vectors $\{\Phi_{a,p}(x):x\in X\}$.
On the other hand, since the image of this feature map is an $M_a$-dimensional feature space, and $M_a<\infty$, the kernel $k_{a,p}$ is not 
strictly positive definite, 
i.e., the matrix $A$ may be singular even for pairwise distinct points $X$. In particular, if $N>M_a$ there exists no set $X\in\R^d$ of $N$ points such that 
the kernel matrix is non singular, since this would require the $N$ vectors $\{\Phi_{a,p}(x):x\in X\}$ to be linearly independent in an $M_a<N$ dimensional 
space.

We recall moreover (see e.g. \cite{Saitoh2016}) that each positive definite kernel $k:\Omega\times\Omega\to\R$ is associated to a RKHS $(\calh, 
\inner{\calh}{})$, which is a 
Hilbert space of functions from $\Omega$ to $\R$, where the kernel acts as a reproducing kernel, i.e., it holds
\begin{itemize}
\item $k(\cdot, x) \in \Omega$ for all $x\in\Omega$
\item $\inner{\calh}{k(\cdot, x), f} = f(x)$ for all $x\in\Omega$ and $f\in\calh$,
\end{itemize}
and that this RKHS is unique given a positive definite kernel $k$ and a set $\Omega$. 

The RKHS of a kernel is often called its native space in the approximation theory literature. We denote as $\calh_{a, p}:=\calh_{a, p}(\Omega)$ the native 
space of $k_{a, p}$ on $\Omega$, and we will discuss its characterization in Section \ref{sec:native_and_error}. 

\begin{remark}
We remark that there are possible extensions to the definition of the polynomial kernel that we use in this section. Most notably, one may consider 
more general sets $\Omega$, not necessarily in the Euclidean space, and replace $\left\langle x,y\right\rangle$ with a corresponding inner product on $\Omega$. 
Although this extension is of potential interest, we do not consider it in this paper.
\end{remark}

\section{Existence and characterization of unisolvent sets}\label{sec:unisolvent}
We start by analyzing conditions on a set of interpolation points that guarantee the existence of a unique polynomial kernel interpolant. 
As recalled in Section \ref{sec:intro}, the kernel interpolant \eqref{eq:kernel_model} exists and is unique whenever the linear system \eqref{eq:kernel_system} 
has a unique solution, i.e., when the kernel matrix of $k_{a,p}$ on $X$ is invertible, i.e., positive definite since it is positive semidefinite by 
construction. We will thus investigate conditions for the invertibility of this kernel matrix.

To derive our characterization we are going to use some relations with interpolation points for classical polynomials. To this end, we recall the following 
definition.
\begin{definition}[Unisolvent set]\label{def:unisol}
Let $\Omega\subset \R^d$ and let $U\subset C(\Omega)$ be a finite dimensional linear space of continuous functions. A set of $N:=\dim(U)$ points $X:=\{x_i\}_{i=1}^N\subset\Omega$ is 
said to be $U$-unisolvent if one of the following equivalent conditions hold:
\begin{enumerate}[(i)]
    \item\label{item:unisol_one} For each $Y:=\{y_i\}_{i=1}^N$ there exists a unique element $u\in U$ which interpolates $Y$ on $X$, i.e, $u(x_i) = y_i$, $1\leq i\leq N$.
    \item\label{item:unisol_two} If $u\in U$ and $u_{|X}=0$, then $u=0$.
    \item\label{item:unisol_three} If $\{u_1, \dots, u_N\}$ is a basis of $U$, then the matrix $V:=V\left(\{u_j\}_{j=1}^N, X\right):=[u_j(x_i)]_{i,j=1}^N\in\R^{N\times N}$ is invertible. 
\end{enumerate}
\end{definition}
We furthermore recall that the matrix $V$ defined in point \eqref{item:unisol_three} is called a Vandermonde matrix if $U=\pol_d^p$ for some $p,d\in\N$, and 
if $\{u_1, \dots, u_N\}$ is any enumeration of the monomial basis of this space.

The fundamental step to derive the results of this section is the following simple lemma, which establishes a connection between the kernel matrix and a 
rectangular Vandermonde matrix.
\begin{lemma}\label{lemma:a_as_vdv}
Let $p\in\N$ and $a\geq 0$, and let $\left\{\zeta^{(i)}\right\}_{i=1}^{M_a}$ be an enumeration of $I_a(p, d)$. 
Let furthermore $X_N\subset\Omega$ be a set of $N\leq M_a$ pairwise distinct points and let $V\in\R^{N\times M}$ be the Vandermonde matrix given by the 
evaluation of the monomials $\{x^{\zeta^{(i)}}\}_{i=1}^{M_a}$ on $X_N$, with columns ordered according to the chosen enumeration, i.e.,
\begin{equation}\label{eq:partial_vdm}
V:= \begin{bmatrix}
    x_1^{\zeta^{(1)}}&\dots & x_1^{\zeta^{({M_a})}}\\
    \vdots & \ddots &\vdots\\
    x_N^{\zeta^{(1)}}&\dots & x_N^{\zeta^{({M_a})}}\\
    \end{bmatrix}.
\end{equation}
Then for all $a>0$ the kernel matrix $A$ of $k_{a,p}$ on $X_N$ satisfies
\begin{equation*}
A = V D V^T,
\end{equation*}
where $D:=\diag\left(d^a_{\zeta^{(1)}}, \dots, d^a_{\zeta^{(M_a)}}\right)$ and $d^a_\zeta$ is defined as in \eqref{eq:d_alpha} for $a>0$ and as in 
\eqref{eq:expansion_kernel_hom} for $a=0$.
\end{lemma}
\begin{proof}
We consider the case $a>0$, since for $a=0$ the argument is the same. 
The result follows by direct computation using the representation \eqref{eq:expansion_kernel} of the kernel. Indeed, using the definition of $V, D$ given in 
the statement we have for all 
$1\leq i, j\leq N$ that
\begin{equation*}
A_{ij} 
= k_{a,p}(x_i, x_j) 
= \sum_{\zeta\in I(p,d)} d_\zeta^a  x_i^\zeta x_j^\zeta
= \sum_{\ell=1}^{M_a} d_{\zeta^{(\ell)}}^a  x_i^{\zeta^{(\ell)}} x_j^{\zeta^{(\ell)}}
= \sum_{\ell=1}^{M_a} d_\ell^a V_{i\ell} V_{j\ell}
=(V D V^T)_{ij},
\end{equation*}
which is the desired representation.
\end{proof}

We furthermore need the following result which shows that, under a certain rank condition, any set of points can be completed to a unisolvent set. Observe that 
for our purposes it is sufficient to prove that the set $X_{M}$ defined in the lemma is in $\R^d$, so we do not put much care in constraining its location into 
a smaller compact subset of $\R^d$.
\begin{lemma}\label{lemma:complete_unisolv}
Let $M,N\in\N$, $N\leq M$, $U:=\Span\{u_i\}_{i=1}^M$ be a linear space of functions in $\R^d$, and let $X_N\subset\R^d$ be such that the matrix $V:=[u_j(x_i)]_{1\leq i\leq N, 1\leq j\leq M}\in\R^{N\times M}$ has full row rank. Then there exist an $U$-unisolvent set $X_M\subset \R^d$ with $X_N\subset X_M$.
\end{lemma}
\begin{proof}
The proof is a simple extension of Lemma 1 in \cite{Bialas2016}. Since $V$ has full row rank there exists $\{j_1,\dots, j_N\}\subset \{1, \dots, M\}$ such that 
$\tilde V:=[V_{i,j_\ell}]_{1\leq i,\ell\leq N}\in\R^{N\times N}$ is invertible, and in particular $\det(\tilde V)\neq 0$. We can then pick any $j_{N+1}\in 
\{1,\dots, M\}\setminus \{j_1,\dots, j_N\}$ and consider the matrix function $V'(x):=V(\{u_{j_\ell}\}_{\ell=1}^{N+1}, X_N\cup\{x\}\}$. Computing the 
determinant of $V'(x)$ by expanding the last column gives 
\begin{equation*}
g(x) := \det(V'(x))) = \sum_{j=1}^{N+1} c_j u_j(x),
\end{equation*}
where $c_{N+1} = \det(\tilde V)\neq 0$ by assumption. It follows that $g\neq 0$ since the $u_j$ are linearly independent, and thus there exists 
an $x_{N+1}\in\R^d$ such that $g(x_{N+1})\neq 0$ (clearly $x_{N+1}\notin X_N$). We can thus define $X_{N+1}:=X_N\cup\{x_{N+1}\}$ and repeat the operation by 
picking another $j_{N+2}$ until $\{j_1,\dots, j_M\}=\{1,\dots, M\}$.
\end{proof}

These lemmas immediately give the first characterization of unisolvency for polynomial kernel interpolation, that is stated in the following proposition. 
Observe that in this case we make a distinction between points in $\Omega$, which is the given domain where the interpolation problem is defined, and point 
which instead may be in $\R^d\setminus \Omega$. 
\begin{proposition}\label{prop:unisolvent_points}
Let $p\in\N$ and $a\geq 0$, and let $X_N\subset\Omega$ be a set of $N\leq M_a$ pairwise distinct points. 
Then the kernel matrix of $k_{a,p}$ on $X_N$ is invertible (and thus positive definite) if and only if the Vandermonde matrix 
\eqref{eq:partial_vdm} has full row rank $N$.

In particular, this is the case if and only if there exists a set $X_{M_a}\subset \Omega$ such that $X_N\subset X_{M_a}$ and $X_{M_a}$ is $\pol_p^d(\Omega)$-unisolvent 
if $a>0$, or $\hpol_p^d(\Omega)$-unisolvent if $a=0$.
\end{proposition}
\begin{proof}
For any $u\in\R^N\setminus\{0\}$ we have by Lemma \ref{lemma:a_as_vdv} that $u^T A u = u^T V D V^T u = v^T D v$, with 
$v:=V^T u\in\R^{M_a}$. Since the matrix $D$ is invertible by construction, we have that $ v^T D v=0$ if and only if $v=0$, i.e., if $V^T 
u=0$, i.e., if and only if the columns of $V^T$ - or the rows of $V$ - are linearly dependent. This proves that in fact $u^T A u\neq 0$ for all 
$u\in\R^N\setminus\{0\}$, i.e., $A$ is positive definite, if and only if the rows of $V$ are linearly independent.

It remains to prove that this condition is equivalent to the existence of a set $X_{M_a}$ that contains $X_N$ and is unisolvent for the corresponding space of polynomials.
If $X_N$ is a subset of a set of unisolvent points $X_{M_a}$ (either for $\pol_p^d$ or $\hpol_p^d$), then the columns of $V$ are clearly linearly independent, 
since $V$ is obtained by selecting $N$ rows from the full Vandermonde matrix $V'$ given by the evaluation of the same monomials on the entire set of points
$X_{M_a}$, and $V'$ is invertible by definition since $X_{M_a}$ is unisolvent for the corresponding space of polynomials.
The converse implication follows instead from Lemma~\ref{lemma:complete_unisolv}.
\end{proof}

As one may expect, the condition of the last proposition is related to the unisolvency of the interpolation set for standard polynomial interpolation.
However, it is remarkable that it is sufficient (and necessary) to have $X_N\subset X_{M_a}$ with $X_{M_a}$ a polynomially unisolvent set (either 
$\pol_p^d$-unisolvent for $a>0$, or $\hpol_p^d$-unisolvent for $a=0$), since this opens the 
possibility to solve polynomial-like interpolation problems with an arbitrary number of interpolation points. This is in contrast with the case of standard 
polynomial interpolation, which requires $N=M_a$ and $M_a$ takes 
only some specific values, depending on $d$ and $p$. Moreover, the construction of Proposition \ref{prop:unisolvent_points} shows that the points 
$X_{M_a}\setminus X_N$ are not bounded to be in $\Omega$, but can be chosen in the entire space $\R^d$.
 
As a consequence of Proposition \ref{prop:unisolvent_points}, we show that the connection with polynomial interpolation is even stronger, i.e., kernel 
interpolation with minimal $p$ is in fact plain polynomial interpolation, either in $\pol_p^d$ or $\hpol_p^d$. We have the following.
\begin{corollary}\label{cor:minimal_p}
Let $p\in\N$ and $a\geq 0$. Assume that $N=M_a$ and that the points $X_N$ are $\pol_p^d$-unisolvent if $a>0$ or $\hpol_p^d$-unisolvent if $a=0$.
Then the polynomial kernel interpolant on $X_N$ coincides with the polynomial interpolant from $\pol_p^d$ if $a>0$ or from $\hpol_p^d$ if $a=0$.
\end{corollary}
\begin{proof}
We have $A=V D V^T$ from Lemma \ref{lemma:a_as_vdv}, and since $M_a=N$ then $V$ is a square invertible matrix. It follows that $A=V B$ with an invertible 
matrix of change of basis $B:=DV^T$, and thus the monomial basis $\{x^{\zeta^{(i)}}\}_{i=1}^{M_a}$ spans the same space of the kernel basis $\{k(\cdot, 
x_i)\}_{i=1}^N$ of \eqref{eq:kernel_model}. In particular the two interpolants coincide by uniqueness (see point \eqref{item:unisol_one} of Definition 
\ref{def:unisol}).
\end{proof}

Finally, we combine Proposition \ref{prop:unisolvent_points} with a construction of polynomial-unisolvent sets in order to derive an explicit characterization 
of 
point sets for interpolation with the polynomial kernel in the case $a>0$.

\begin{theorem}\label{thm:existence}
Let $a> 0$ and let $X_N\subset\Omega$ be a set of $N$ pairwise distinct points. 
Then for any $p\geq d (N-1)$ there exists a set $X_{M_p^d}\subset \Omega$ of $M_p^d$ points such that $X_N\subset X_{M_p^d}$ and $X_{M_p^d}$ is 
$\pol_p^d$-unisolvent. 
In particular, the kernel matrix of $k_{a,p}$ on $X_N$ is invertible if $p\geq d (N-1)$.
\end{theorem}
\begin{proof}
By virtue of \cite[Theorem 1]{Chung1977}, it is sufficient for $X_{M_p^d}$ to satisfy the following geometric property: For each $x_i\in X_{M_p^d}$, there 
exist $p$ distinct hyperplanes $G_{i1},\dots,G_{ip}$ such that
\begin{enumerate}
    \item 
    $x_i$ does not lie on any of these hyperplanes;
    \item
    all the other nodes $X_{M_p^d}\setminus\{x_i\}$ lie on at least one of the hyperplanes.
\end{enumerate}
In order to construct $X_{M_p^d}$, we can proceed as follows. For each $x_i\in X_{N}$, let $L_{i_1},\dots,L_{i_d}$ be $d$ hyperplanes such that 
$L_{i_1}\cap\dots\cap L_{i_d}=\{x_i\}$ and $x_j\notin L_{i_1}\cup\dots\cup L_{i_d}$ for $j\neq i$ with $j,i=1,\dots,N$. Moreover, such $Nd$ hyperplanes are 
chosen to be pairwise non-parallel. This construction is always feasible. To get a better intuition, it is possible to reason iteratively. With $x_i\in X_{N}$ fixed, consider $d$ hyperplanes that do not intersect any point in $X_N$ besides $x_i$, thus they are not parallel and moreover choose them such that they do not contain any line connecting $x_i$ with another of the other points. Then, consider another point $x_j\in X_{N}$, $j\neq i$. We can again choose $d$ hyperplanes that intersect at $x_j$ and do not contain the lines connecting $x_j$ with the other points in $X_N$ and are not parallel with the other previously chosen. This is possible due to the infinitely many directions of the hyperplanes passing through $x_j$. Furthermore, we can surely select pairwise non-parallel hyperplanes (if two hyperplanes turn out to be parallel, we can simply rotate one of the two by using the corresponding node as center of rotation and avoid the parallelism). 

Then, proceeding with the proof, the considered hyperplanes intersect at $\binom{Nd}{d}$ points that satisfy the mentioned geometric property. Then, 
by setting $p=Nd-d=d(N-1)$, we have $\binom{Nd}{d}=M_p^d$ and $X_{M_p^d}$ consists of such points.

Moreover, we observe that it is possible to construct $X_{M_p^d}$ for any $p= Nd-d+j$, $j\ge 1$, by adding to the $Nd$ hyperplanes further $j$ pairwise 
non-parallel ones that intersect the other hyperplanes in distinct new points, taking then the resulting intersection points. This proves that the same 
construction works for any $p\geq d(N-1)$.

Finally, Proposition \ref{prop:unisolvent_points} proves that $X_N$ can be used for interpolation with the polynomial kernel $k_{a,p}$ if $p\geq d(N-1)$.
\end{proof}

We remark that this theorem has the notable implication that, if $a>0$, any set of pairwise distinct points can be used for interpolation with $k_{a,p}$ 
provided that $p$ is large enough.

\begin{remark}
It remains open
% \todo{Oppure lo sappiamo fare? FM: non penso}
to investigate if this construction works also for $a=0$, i.e., for the space of homogenous polynomials. 
Moreover, it would be interesting to investigate if the lower bound $p\geq d(N-1)$ provided by the theorem is sharp. For now, we observe that this is clearly 
optimal for $d=1$, since in this case the theorem implies that any $p\geq N-1$ can be used, or in fact that no point needs to be added. This corresponds to the 
fact that any set of pairwise distinct points is unisolvent for polynomial interpolation in $\R$.
\end{remark}

\section{Interpolation, stability, and error estimation}\label{sec:native_and_error}
Now that the existence of interpolants is established we want to understand the corresponding approximation behavior. 

In the last section we dealt with interpolation problems in terms of generic target data $y\in\R^N$. From now on, we will additionally assume that there exists 
a continuous function $f\in C(\Omega)$ such that $y_i:=f(x_i)$, $1\leq i\leq N$. In this case, we will denote the $k_{a,p}$-interpolant \eqref{eq:kernel_model} 
as $I_{X,a,p} f$. From the representation \eqref{eq:kernel_model}, it is immediate to see that $I_{X,a,p} f$ is an element of the linear subspace 
$V_{a,p}(X):=\Span\{k_{a,p}(\cdot, x):x\in X\}\subset \calh_{a,p}$, and thus $I_{X,a,p}$ can be understood as a map from $C(\Omega)$ (or $\calh_{a,p}$) to 
$V_{a,p}(X)$. For this reason, we are interested in obtaining more explicit information on the native space $\calh_{a,p}(\Omega)$ on $\Omega\subset\R^d$, which 
will connect it to suitable spaces of polynomials. 

In any case, we underline that in order to obtain asymptotic stability or convergence estimates, one would like to consider an increasing number $N$ of 
interpolation points. To ensure the existence of the interpolant, one thus needs to consider a kernel $k_{a,p}$ with an $N$ (or $X$) dependent parameter $p$, 
and the design of an optimal choice of $p$ to have stability and convergence is an interesting question. In the following sections we will try to connect 
the stability and accuracy of kernel interpolation to polynomial interpolation, which seems a promising way to address these issues.

\begin{remark}\label{rem:comments_on_existence}
In contrast with polynomial interpolation, the interpolation space $V_{a,p}(X)$ is depending on $X$. In the language of Definition~\ref{def:unisol}, we 
should thus say that $X$ allows unique interpolation by $k_{a,p}$ if and only if the points $X$ are $V_{a,p}(X)$-unisolvent. 
To simplify the presentation, in the following we will instead simply write that the set $X$ is $\calh_{a,p}$-unisolvent.
\end{remark}

\subsection{Native space}
We first derive some characterization of the space $\calh_{a,p}$ for $a\geq 0$ and $p\in\N$, and to this end we introduce some additional notation. Given 
$\gamma\in\N_0^d$ with $|\gamma|\leq p$, and $f\in C^p(\Omega)$, we write 
\begin{equation*}
D^\gamma f(x):=\prod_{i=1}^d \left(\partial_{x^{(i)}}^{\gamma^{(i)}} f(x)\right)
\end{equation*}
for the derivative of $f$ with multi-index $\gamma$. 
For all $a\geq 0$ and $\zeta\in I_a(p, d)$, we furthermore define the weights
\begin{equation}\label{eq:def_w}
w_\zeta^a:= (\zeta!)^2 d_\zeta^a=
\begin{cases}
\frac{p! \zeta!}{(p-|\zeta|)!} a^{p-|\zeta|}, & a>0,\\
p! \zeta!, & a=0,\\
\end{cases}
\end{equation}
so that \eqref{eq:expansion_kernel} and \eqref{eq:expansion_kernel_hom} give
\begin{equation*}
k_{a,p}(x, y) 
= \sum_{\zeta\in I_a(p,d)} d_\zeta^a  x^\zeta y^\zeta
= \sum_{\zeta\in I_a(p,d)} w_\zeta^a \frac{x^\zeta}{\zeta!} \frac{y^\zeta}{\zeta!}.
\end{equation*}
For $\gamma\in I_a(p,d)$ and $x\in\Omega$ this implies that
\begin{equation}\label{eq:kernel_der}
\left(D^\gamma_x k_{a,p}(x, y)\right)_{x=0}
= \sum_{\zeta\in I_a(p,d)} w_\zeta^a  \left(D^\gamma \frac{x^\zeta}{\zeta!} \right)_{x=0} \frac{y^\zeta}{\zeta!} =  w_\gamma^a \frac{y^\gamma}{\gamma!},
\end{equation}
since $\left(D^\gamma x^\zeta\right)_{x=0}=0$ if $\zeta\neq\gamma$, while $D^\gamma x^\gamma/\gamma!=1$.

With these observations we can prove the following result. The characterization of $\calh_{a,p}$ is a special case of the argument of Section 2 in 
\cite{Zwicknagl2009} that we repeat for completeness.
\begin{theorem}\label{thm:ns_characterization}
The native space of $k_{a,p}$ on $\Omega$ is $\calh_{a, p}(\Omega)=\pol_p^d(\Omega)$ if $a>0$ and $\calh_{0, p}=\hpol_p^d(\Omega)$ if $a=0$, with the inner 
product
\begin{equation}\label{eq:inner_product}
  \inner{\calh_{a,p}}{f, g}:=\sum_{\gamma\in I_a(p, d)} \frac{1}{w_\gamma^a} D^\gamma f(0) D^\gamma g(0),\;\;f,g\in\calh_{a,p}.
\end{equation}
\end{theorem}
\begin{proof}
We prove the result for $a>0$, since if $a=0$ the same reasoning applies.
Namely, the expression \eqref{eq:inner_product} clearly defines a symmetric, bilinear, and positive definite form on $\pol_p^d(\Omega)$, and thus an inner 
product. 
Since $\pol_p^d(\Omega)$ is finite it follows that $\left(\pol_p^d(\Omega), \langle \cdot, \cdot\rangle_{\calh_{a,p}}\right)$ is a Hilbert space.
If we prove that $k_{a,p}$ is a reproducing kernel on this space then it must hold  that $\calh_{a,p}(\Omega)=\pol_p^d(\Omega)$ by uniqueness of the native 
space of a given kernel (see Section \ref{sec:background}). 

This is indeed the case since $k_{a,p}(\cdot, x)\in \pol_p^d(\Omega)$ for all $x\in\Omega$ thanks to \eqref{eq:expansion_kernel}. Moreover, for all 
$f:=\sum_{\beta\in I_a(p,d)}c_\beta x^\beta\in\pol_p^d(\Omega)$ it holds $D^\gamma f(0) = \gamma! c_\gamma$, and thus for all $x\in\Omega$ using 
\eqref{eq:kernel_der} we have 
\begin{equation*}
\left\langle f, k_{a,p}(\cdot, y)\right\rangle_{\calh_{a,p}} 
= \sum_{\gamma\in I_a(p,d)} \frac{1}{w_\gamma^a} D^\gamma f(0) \left(D^\gamma_x k_{a,p}(x, y)\right)_{x=0}
= \sum_{\gamma\in I_a(p,d)} \frac{1}{w_\gamma^a} \gamma! c_\gamma w_\gamma^a \frac{y^\gamma}{\gamma!}
= \sum_{\gamma\in I_a(p,d)} c_\gamma y^\gamma = f(y),
\end{equation*}
which is the reproducing property of the kernel. We thus have that $k_{a,p}$ is a reproducing kernel and the first part of the theorem is proven. The same 
argument works for $a=0$ and $\hpol_p^d$ using \eqref{eq:expansion_kernel_hom} in place of \eqref{eq:expansion_kernel}.
\end{proof}

This characterization makes it possible to study in an explicit manner the effect of the parameters $a,p$ on the native spaces, and thus on the corresponding 
approximants.
\begin{corollary}\label{cor:norm_equiv}
For any $a\geq 0$, $p\in\N$ we have the following.
\begin{enumerate}[(i)]
\item\label{item:equiv} 
If $0<a'\leq a$ then the native spaces of $k_{a,p}$ and $k_{a', p}$ are norm equivalent, i.e., $\calh_{a, p}=\calh_{a', p}$ as sets and
\begin{equation}\label{eq:norm_equiv}
\left(a'/a\right)^{p/2} \left\|f\right\|_{\calh_{a', p}}
\leq 
\left\|f\right\|_{\calh_{a, p}}
\leq 
\left\|f\right\|_{\calh_{a', p}} \;\;\fa f\in \calh_{a, p},
\end{equation}
while for any $a>0$ it holds $\calh_{0,p}\subset \calh_{a, p}$ with $\Norm{\calh_{a, p}}{f}=\Norm{\calh_{0, p}}{f}$ for all $f\in \calh_{0, p}$.

\item\label{item:equiv_pq} If $a>0$ and $p,q\in\N$ with $0\leq q\leq p$, we have that $\calh_{a, q}\subset\calh_{a, p}$, and the norms of the two spaces are 
equivalent on $\calh_{a, q}$ with
\begin{equation}\label{eq:norm_equiv_pq}
a^{(q-p)/2} \left\|f\right\|_{\calh_{a, p}}
\leq 
\left\|f\right\|_{\calh_{a, q}}
\leq 
a^{(q-p)/2} \binom{p}{p- q}^{1/2} \left\|f\right\|_{\calh_{a, p}}
\;\;\fa f\in \calh_{a, q}.
\end{equation}
\end{enumerate}
\end{corollary}
\begin{proof}
For the first point we have clearly $\calh_{a, p}=\calh_{a', p}=\pol_p^d$ as sets. Moreover for all 
$\zeta\in I_a(p, d)$ we have
\begin{equation*}
w_\zeta^a
= \frac{p! \zeta!}{(p-|\zeta|)!} a^{p-|\zeta|} 
= \left(\frac{a}{a'}\right)^{p-|\zeta|}  w_\zeta^{a'},
\end{equation*}
and thus
\begin{equation}\label{eq:w_bounds}
w_{\zeta}^{a'}
\leq 
w_\zeta^a 
\leq 
\left(a/a'\right)^{p}  w_\zeta^{a'},
\end{equation}
since $a/a'\geq 1$. Using this relation in the definition \eqref{eq:inner_product} of the norm implies that
\begin{align*}
\left(\frac{a'}{a}\right)^{p/2} \left\|f\right\|_{\calh_{a', p}}
\leq 
\left\|f\right\|_{\calh_{a, p}}
\leq 
\left\|f\right\|_{\calh_{a', p}},
\end{align*}
which is \eqref{eq:norm_equiv}. 
Finally $\calh_{0, p}=\hpol_p^d\subset \pol_p^d = \calh_{a,p}$ for all $a>0$, and so it makes sense to compute the $\calh_{a,p}$-norm of $f(x)=\sum_{\zeta\in 
I_0(p,d)}c_{\zeta} x^\zeta\in \calh_{0,p}$. Since $D^\gamma f(0) = 0$ for all $\gamma\in \N_0^d$ with $|\gamma|<p$,
and since $w_\gamma^a= w_{\gamma}^0$ if $|\gamma|=p$ (see \eqref{eq:def_w}), we have
\begin{align*}
\Norm{\calh_{a,p}}{f}^2 
= \sum_{\gamma\in I_a(p, d)} \frac{1}{w_\gamma^a} \left(D^\gamma f(0)\right)^2
= \sum_{\gamma\in I_0(p, d)} \frac{1}{w_\gamma^a} \left(D^\gamma f(0)\right)^2
= \sum_{\gamma\in I_0(p, d)} \frac{1}{w_\gamma^0} \frac{a^{p-|\gamma|}}{(p-|\gamma|)!} \left(D^\gamma f(0)\right)^2
=\Norm{\calh_{0,p}}{f}^2,
\end{align*}
and this concludes the proof of the first part.

In the second case the space inclusion is also clear, and to prove the norm equivalence we write $w_{\zeta}^a(p)$, $w_{\zeta}^a(q)$ with an explicit 
dependence on the polynomial degrees $p,q$.  
For $f\in \calh_{a,q}$ we have by definition of the norm of $\calh_{a,p}$ that
\begin{equation}\label{eq:equiv_tmp}
\Norm{\calh_{a,p}}{f}^2 
= \sum_{\zeta\in I_a(p, d)} \frac{1}{w_\zeta^a(p)} \left(D^\zeta f(0)\right)^2
= \sum_{\zeta\in I_a(q, d)} \frac{1}{w_\zeta^a(p)} \left(D^\zeta f(0)\right)^2,
\end{equation}
where the second equality follows from the fact that the derivatives vanish for $|\zeta|>q$. Moreover for all $\zeta \in I_a(q, d)$ we have
\begin{equation*}
w_\zeta^a(p)
= \frac{p! \zeta!}{(p-|\zeta|)!}a^{p-|\zeta|}
= \frac{p\cdots (q+1) q! \zeta!}{(p-|\zeta|)\cdots(q+1-|\zeta|)(q-|\zeta|)!}a^{q-|\zeta|}a^{p-q}
= \frac{p\cdots (q+1)}{(p-|\zeta|)\cdots(q+1-|\zeta|)}a^{p-q} w_\zeta^a(q),
\end{equation*}
and this quantity is minimized when $|\zeta|=0$ and maximized when $|\zeta|=q$, giving
\begin{equation*}
a^{p-q} w_\zeta^a(q)
\leq 
w_\zeta^a(p)
\leq  
a^{p-q} \binom{p}{p-q} w_\zeta^a(q).
\end{equation*}
Inserting these bounds in  \eqref{eq:equiv_tmp} gives 
\begin{equation*}
a^{q-p} \binom{p}{p- q}^{-1} \left\|f\right\|_{\calh_{a, q}}^2
\leq 
\left\|f\right\|_{\calh_{a, p}}^2
\leq 
a^{q-p}  \left\|f\right\|_{\calh_{a, q}}^2
\;\;\fa f\in \calh_{a, q},
\end{equation*}
which can be rearranged to obtain \eqref{eq:norm_equiv_pq}.
\end{proof}

\begin{remark}
The inclusion relation $\calh_{a,p}\subset\calh_{a,p'}$ for $p\leq p'$ and $a=1$, and the equivalence between the corresponding norms, was 
already proven in \cite{ZhZh2013b} by other arguments (see Proposition 4.3 and Proposition 6.3). The other relations are instead new to the best of our 
knowledge. Moreover, the case \eqref{item:equiv_pq} {cannot} be extended to $a=0$, since in this case the spaces $\calh_{0, p}$ and $\calh_{0,p-q}$ have empty 
intersection unless $q=p$.
Finally, we remark that the general case $a\neq a'$, $p\neq q$ can be obtained by combining the two cases \eqref{item:equiv} and \eqref{item:equiv_pq} of 
Corollary \ref{cor:norm_equiv}.
\end{remark}

We point out that \eqref{eq:inner_product} implies that the monomials are orthogonal in $\calh_{a,p}$, i.e., for all $\zeta, \beta\in I_a(p,d)$ we have
\begin{equation}\label{eq:monomial_inner}
\inner{\calh_{a,p}}{\frac{x^\zeta}{\zeta!}, \frac{x^\beta}{\beta!}} = \frac{1}{w_\zeta^a} \delta_{\zeta\beta}.
\end{equation}
From this fact we may also deduce that the norm inequality \eqref{eq:norm_equiv} is sharp for all $0<a'\leq a$. Indeed, the right inequality in 
\eqref{eq:norm_equiv} is an equality for $f(x):=x^\zeta$ with $|\zeta|=p$, since in this case we have by \eqref{eq:monomial_inner} that 
\begin{equation*}
\Norm{\calh_{a,p}}{x^\zeta}^2
= \frac{(\zeta!)^2}{w_\zeta^a}
= \frac{\zeta!}{p!} 
=\frac{(\zeta!)^2}{w_\zeta^{a'}}
=\Norm{\calh_{a',p}}{x^\zeta}^2, 
\end{equation*}
where we used the definition \eqref{eq:def_w} of $w_\zeta^a$ and $w_\zeta^{a'}$. 
Similarly, the left inequality in \eqref{eq:norm_equiv} is met for $|\zeta|=0$, i.e. $x^\zeta=1$ and $\zeta!=1$, since in this case we have again by 
\eqref{eq:monomial_inner} that
\begin{equation*}
\Norm{\calh_{a,p}}{x^\zeta}^2
= \frac{1}{w_\zeta^a}
= \frac{1}{a^p} 
= \left(\frac{a'}{a}\right)^p \frac{1}{w_\zeta^{a'}}
= \left(\frac{a'}{a}\right)^p\Norm{\calh_{a',p}}{x^\zeta}^2. 
\end{equation*}
In particular, maximal-degree monomials have the same norm independently of $a>0$, while lower degree monomials have an increasingly large norm as $a\to 0$,  
up 
to not even being elements of $\calh_{a,p}$ in the limiting case, since indeed $\calh_{0,p}$ is the homogeneous space. In this sense, the parameter $a\geq 0$ 
has a regularizing effect, promoting high degree components in a minimal norm solution of the interpolation problem.

For \eqref{eq:norm_equiv_pq}, similar arguments prove that the equality is obtained for $|\zeta|=0$ (the lower bound), and $|\zeta|=q$ (the upper 
bound). It thus happens that if $p\geq q$, the elements of $\calh_{a,q}$ have an $\calh_{a,p}$-norm which increases with a factor 
between $a^{(q-p)/2}$ and $a^{(q-p)/2} \binom{p}{p-q}^{1/2}$. In particular, low-degree and high-degree monomials have norm that are increasingly 
separated, and thus also increasing $p$ has a regularization effect, and the monomials $x^{\zeta}$ have minimal norm in $\calh_{a,p}$ with $p=|\zeta|$.

\subsection{Stability}\label{sec:stability} 
We now derive a simple stability results for the interpolation map $I_{X,a,p}:C(\Omega) \to V_{a,p}(X)$, as a function of a set $X\subset \Omega$ of 
$\calh_{a,p}$-unisolvent points.
We recall that the Lebesgue function associated to the interpolation process is defined as
\begin{equation}\label{eq:lebesgue_fun}
    \lambda_{X,a,p}(x):=\sup\limits_{0\neq f\in V_{a,p}(X)}\frac{|f(x)|}{\;\Norm{\infty}{f_{|X}}},\;\;x\in \Omega,
\end{equation}
such that one obtains the stability bound
\begin{equation}\label{eq:lebesgue_fun_bound}
|I_{X,a,p} f(x)|
\leq \lambda_{X,a,p}(x) \Norm{\infty}{f_{|X}}
\leq \lambda_{X,a,p}(x) \Norm{L_\infty(\Omega)}{f}
\;\; \fa f\in C(\Omega), x\in \Omega,
\end{equation}
which can be also written in terms of the associated Lebesgue constant $\Lambda_{X,a,p}:=\Norm{L_\infty(\Omega)}{\lambda_{X,a,p}}$ as
\begin{equation}\label{eq:stab_bound} 
\Norm{L_\infty(\Omega)}{I_{X,a,p} f} 
\leq \Lambda_{X,a,p} \Norm{\infty}{f_{|X}}
\leq \Lambda_{X,a,p} \Norm{L_\infty(\Omega)}{f}
\;\; \fa f\in C(\Omega).
\end{equation}
Moreover, the fact that $X$ is unisolvent ensures the existence of a Lagrange basis $\left\{\ell_{i,a,p}\right\}_{i=1}^N$ of $V_{a,p}(X)$, which gives
\begin{equation}\label{eq:cardinal_interpolant}
I_{X,a,p} f(x) = \sum_{i=1}^N f(x_i) \ell_{i, a,p}(x), \;\;x\in\Omega,
\end{equation}
and $\lambda_{X,a,p}(x)=\sum_{i=1}^N \left|\ell_{i,a,p}(x)\right|$.

Although we are still not able to obtain explicit bounds on $\lambda_{X,a,p}$ and $\Lambda_{X,a,p}$, we can prove the following result.

\begin{theorem}\label{thm:stability}
Let $B\subset\R^d$ be a set and let $X\subset\Omega$ be a set of $\pol_p^d$-unisolvent points. Let 
\begin{equation}
\Lambda_{X}^{pol}:=\sup\limits_{0\neq f\in \pol_p^d(B)}\frac{ \Norm{L_\infty(B)}{f}}{\;\Norm{\infty}{f_{|X}}}
\end{equation}
be the Lebesgue constant for polynomial interpolation of degree $p$ on $X$.

Let $a\geq 0$, $p\in\N$, and let $X$ be $\calh_{a,p}$-unisolvent. For any $X_{M_a}\supset X$ which is $\pol_p^d$-unisolvent, it holds
\begin{equation*}
\Norm{L_\infty(\Omega)}{I_{X,a,p} f} 
\leq \Lambda_{X_{M_a}}^{pol} \Norm{L_\infty(B)}{f}
\;\; \fa f\in C(B),
\end{equation*}
where $B\subset\R^d$ is any set which contains $X_{M_a}$.
\end{theorem}
\begin{proof}
The results simply follows from the fact that $I_{X,a,p} f\in \calh_{a,p}=\pol_{p}^d$ and by the definition of $\Lambda_{X_{M_a}}^{pol}$. Such $X_{M_a}$ exists 
thanks to Proposition \ref{prop:unisolvent_points}, but since it needs not to be contained in $\Omega$ we consider an enclosing set $B$. 
\end{proof}

This results shows that polynomial kernel interpolation on $X$ is at least as stable as polynomial interpolation on any $X_{M_a}\supset X$. This points 
to the fact that understanding how to complete a set of points to a set of polynomially unisolvent points of small Lebesgue constant may be highly relevant in 
this context. The fact is also related to the possible minimality of the construction in Theorem \ref{thm:existence} (see also Remark 
\ref{rem:comments_on_existence}). Moreover, although choosing $X_{M_a}\setminus X_N$ outside of $\Omega$ may possibly lead to a smaller Lebesgue constant, one 
pays the price of obtaining a bound in terms of the norm $\Norm{L_\infty(B)}{f}$ computed on a larger set $B\supset \Omega$, which may possibly be significantly 
larger than $\Norm{L_\infty(\Omega)}{f}$.

On the other hand, if one is free to choose $N$ points to sample a function to construct a $k_{a,p}$ interpolant, this result suggests that it could be a good 
idea to select them from a $\pol_p^d$-unisolvent set with small Lebesgue constant.

\subsection{Error estimation}
As it is typically the case in kernel interpolation, we start by assuming that $f\in\calh_{a,p}$. In this case, we recall 
that the interpolant as a map $I_{X,a,p}:\calh_{a,p}\to V_{a,p}(X)$ coincides with the $\calh_{a,p}$-orthogonal 
projection onto $V_{a,p}(X)$. The norm of the associated error operator is the power function $P_{X,a,p}(x)$ defined by
\begin{equation}\label{eq:pf_def}
P_{X,a,p}(x)
:=\sup\limits_{0\neq f\in \calh_{a, p}}\frac{\left|f(x) - I_{X,a,p} f(x)\right|}{\Norm{\calh_{a,p}}{f}}.
\end{equation}
By definition, the interpolation error can be controlled as
\begin{equation}\label{eq:pf_bound}
\left|f(x) - I_{X,a,p} f(x)\right|\leq P_{X,a,p}(x){\Norm{\calh_{a,p}}{f}}\;\; \fa f\in\calh_{a,p}.  
\end{equation}
This bound allows one to separate the error in a term depending only on $f$ and one depending only on $X$, $\Omega$ and $k_{a,p}$. 

In particular, worst-case error bounds in $\calh_{a,p}$ can be derived by obtaining uniform bounds on $P_{X,a,p}$ in terms of the fill distance
\begin{equation*}
h_X:=\sup\limits_{x\in\Omega}\min\limits_{y\in X} \|x-y\|.
\end{equation*}
We refer to Chapter 11 in \cite{Wendland2005} for details on this approach. 

We are not able to obtain bounds of this type yet, and we rather use the power function to outline a method 
to derive more general error bounds.
Indeed, as mentioned in Section \ref{sec:stability} the interpolant $I_{X,a,p} f$ is well defined also for any $f\in C(\Omega)$, even for $f\notin \calh$, 
since its computation requires only the knowledge of $f_{|X}$, and it is of interest to study the resulting approximation error also in this case. We have the 
following result.

\begin{proposition}\label{prop:error}
Let $X_N\subset\Omega$ be $\calh_{a,p}$-unisolvent. Then for all $f\in C(\Omega)$ we have
\begin{equation}\label{eq:error}
\left|(f-I_{X,a,p} f)(x)\right| 
\leq \left(1 + \lambda_{X,a,p}(x)\right)\left\|f-f_p^\star\right\|_{L_\infty(\Omega)} + P_{X,a,p}(x) \left\|f^\star - I_{X,a,p} 
f_p^\star\right\|_{\calh_{a,p}}, \;\;x\in\Omega, 
\end{equation}
where 
\begin{equation*}
f_p^\star:=\inf\limits_{g\in \pol_p^d(\Omega)}\Norm{L_\infty(\Omega)}{f-g}
\end{equation*}
is the uniform best polynomial approximant of $f$.

\end{proposition}
\begin{proof}
We have 
\begin{equation}\label{eq:add_subtr}
f-I_{X,a,p} f = f - f_p^\star + f_p^\star - I_{X,a,p} f_p^\star + I_{X,a,p} f_p^\star - I_{X,a,p} f.
\end{equation}
Using the form 
\eqref{eq:cardinal_interpolant} of the interpolant, a standard argument gives
\begin{align}\label{eq:tmp_bound_one}
|(I_{X,a,p} f_p^\star - I_{X,a,p} f)(x)|
&=|I_{X,a,p} (f_p^\star - f)(x)|
= \sum_{i=1}^N \left|\ell_{i,a,p}(x)\right| |(f_p^\star - f)(x_i)| \leq \lambda_{X,a,p}(x) \max_{1\leq i\leq N} |(f_p^\star-f)(x_i)|\\
&\leq \lambda_{X,a,p}(x) \left\|f^\star-f\right\|_{L_\infty(\Omega)}\nonumber.
\end{align}
Moreover, since $f_p^\star\in \calh_{a,p}=\pol_{p}^d$ by definition, we can apply the standard power function bound \eqref{eq:pf_bound} and get
\begin{equation}\label{eq:tmp_bound_two}
|(f_p^\star - I_{X,a,p} f_p^\star)(x)| \leq P_{X,a,p}(x) \left\|f_p^\star - I_{X,a,p} f_p^\star\right\|_{\calh_{a,p}}.
\end{equation}
Inserting these two bounds in \eqref{eq:add_subtr} we get
\begin{align*}
|(f-I_{X,a,p} f)(x)|
&\leq  |(f - f_p^\star)(x)| + |(f_p^\star - I_{X,a,p} f_p^\star)(x)| + |(I_{X,a,p} f_p^\star - I_{X,a,p} f)(x)|\\
&\leq  \Norm{L_\infty(\Omega)}{f - f_p^\star} + |(f_p^\star - I_{X,a,p} f_p^\star)(x)| + |(I_{X,a,p} f_p^\star - I_{X,a,p} f)(x)|\\
&\leq \left(1 + \lambda_{X,a,p}(x)\right) \left\|f - f_p^\star\right\|_{L_\infty(\Omega)} + P_{X,a,p}(x) \left\|f_p^\star - I_{X,a,p} 
f_p^\star\right\|_{\calh_{a,p}},
\end{align*}
which is the bound of the statement.
\end{proof}
This result shows that it is possible to obtain error bounds for interpolation of functions outside of the native space (thus escaping the native space 
\cite{Narcowich2006}, or working in the so-called misspecified setting \cite{Kanagawa2016b, Kanagawa2019}). 

Observe also that in a sense inequality \eqref{eq:error} is sharp with respect to the relation between polynomial interpolation 
in $\pol_p^d$ and kernel interpolation in $\calh_{a,p}$. Indeed, on one hand if $f\in \calh_{a,p}=\pol_p^d$ then $f-f_p^\star=0$, and thus \eqref{eq:error} 
reduces to the usual power function bound \eqref{eq:pf_bound} for kernel interpolation. On the other hand, if instead $N=M_a$ then $I_{X,a,p} f$ coincides with 
the polynomial interpolant of $f$ (see Corollary \ref{cor:minimal_p}). It follows that $f-I_{X,a,p} f = 0$ for all $f\in \calh_{a,p}=\pol_p^d$, and thus 
$P_{X,a,p}=0$ (see \eqref{eq:pf_def}). In this case \eqref{eq:error} reduces to the standard Lebesgue function bound for polynomial interpolation.

Between these two limits, the proposition suggests as well that one may try to optimize the set $X_{M_a}$ and its subset $X_N$, in order to balance the 
contribution of the two terms in \eqref{eq:error}. 

Moreover, we have the following.
\begin{corollary}
Under the assumptions of Proposition \ref{prop:error}, let $X_{M_a}$ with $X_N\subset X_{M_a}\subset B\subset \R^d$ be any $\pol_p^d$-unisolvent set. Then
\begin{equation}\label{eq:error_vs_pol}
\left|(f-I_{X,a,p} f)(x)\right| 
\leq \left(1 + \lambda_{X_{M_a}}^{pol}(x)\right)\left(\left\|f-f_p^\star\right\|_{L_\infty(B)} + \max\limits_{x\in X_{M_a}}P_{X,a,p}(x) 
\left\|f^\star - I_{X,a,p} f_p^\star\right\|_{\calh_{a,p}}\right), \;\;x\in B, 
\end{equation}
where $\lambda_{X_{M_a}}^{pol}(x)$ is the Lebesgue function for polynomial interpolation.
\end{corollary}
\begin{proof}
We follow the same steps as in the proofs of Proposition \ref{prop:error}. Instead of \eqref{eq:tmp_bound_one}, since $f_p^\star - I_{X,a,p} 
f_p^\star\in\pol_p^d$ we have 
\begin{equation*}
|(I_{X,a,p} f_p^\star - I_{X,a,p} f)(x)|
\leq \lambda_{X_{M_a},a,p}(x) \left\|f^\star-f\right\|_{L_\infty(B)},
\end{equation*}
with the same argument as in Theorem \ref{thm:stability}. Moreover, again because $f_p^\star - I_{X,a,p} f_p^\star\in\pol_p^d$ we have 
\begin{equation*}
|(f_p^\star - I_{X,a,p} f_p^\star)(x)| 
\leq \lambda_{X_{M_a}}^{pol}(x)\Norm{\infty}{(f_p^\star - I_{X,a,p} f_p^\star)_{|X_{M_a}}},
\end{equation*}
and from this we can proceed to replace \eqref{eq:tmp_bound_two} with
\begin{equation*}
|(f_p^\star - I_{X,a,p} f_p^\star)(x)| 
\leq \lambda_{X_{M_a}}^{pol}(x) \max\limits_{x\in X_{M_a}} P_{X,a,p}(x) \left\|f_p^\star - I_{X,a,p} f_p^\star\right\|_{\calh_{a,p}}.
\end{equation*}
Inserting these two bounds in \eqref{eq:add_subtr} we get the result.
\end{proof}
Without the term involving the power function, inequality \eqref{eq:error_vs_pol}
is the error bound for polynomial interpolation of $f$ on $X_{M_a}$, and we thus have that the error of kernel interpolation on $X$ is comparable with 
the error of polynomial interpolation on $X_{M_a}$. To quantify this relation it would be sufficient to prove bounds on $\max_{x\in 
X_{M_a}}P_{X,a,p}(x)$, which should be expected to be much easier to bound than $\Norm{L_\infty(\Omega)}{P_{X,a,p}}$ and with a smaller value, provided the 
points $X_{M_a}$ are not too far from $\Omega$. Moreover, one may expect that such a bound depends on the relation between $X_N$ and $X_{M_a}$, instead of on 
$h_X$.

\begin{remark}
Other approaches may be followed to obtain error bounds for interpolation with the polynomial kernels. Most notably, one can use the zero lemma 
of \cite{Narcowich2004}. Namely, for $\tau:=k+s$ with $k\in\N$, $k>d/2$, $s\in (0,1]$, and $1\leq t\leq \infty$, we denote as $W_t^\tau(\Omega)$ the 
$L_t(\Omega)$-Sobolev space of fractional smoothness $\tau$. Theorem 2.12 in \cite{Narcowich2004} proves that if $\Omega$ has a sufficiently smooth boundary, 
then 
there is a constant $C_k$ depending on $k$ such that for any $1\leq r\leq \infty$ and for any $f\in W_t^\tau(\Omega)$ it holds
\begin{equation*}
\Norm{W_r^\tau(\Omega)}{f - I_{X,a,p} f}
\leq C_k h_X^{\tau - d(1/t-1/r)_+} \seminorm{W_t^\tau(\Omega)}{f - I_{X,a,p} f},
\end{equation*}
where $(x)_+:=\max(x, 0)$. 
Taking in particular $\tau=k+s>p$ gives thus
\begin{equation*}
\Norm{W_r^\tau(\Omega)}{f - I_{X,a,p} f}
\leq 
C_k h_X^{\tau - d(1/t-1/r)_+} \seminorm{W_t^\tau(\Omega)}{f},
\end{equation*}
where $\seminorm{W_t^\tau(\Omega)}{I_{X,a,p} f} =0$ since $\tau>p$.
However, $p$ needs to increase with $N$ (see the beginning of Section \ref{sec:native_and_error}), and thus $\tau=k+s>p$ is possible only if $k$ is itself 
increasing. Since $C_k$ is increasing with $k$, to use this approach one would need to work out explicitly the growth of $C_k$ in terms of $p$, with a small as 
possible $p$. 
A similar approach has already been followed in \cite{Zwicknagl2009}, but in that case the kernels are strictly positive definite, so one does not need to 
change the kernel depending on the points.
\end{remark}

\section{Stable computations}\label{sec:maybe_rbf_qr}
As we will demonstrate numerically in Section \ref{sec:numerics}, the computation of a $k_{a,p}$-interpolant may be significantly unstable from a computational 
point of view, even if we proved in Section \ref{sec:stability} that the Lebesgue constant of this interpolant can be growing quite slowly, for example 
with a rate comparable to that of polynomial interpolation (see Theorem \ref{thm:stability}).

The fact that the interpolation process is provably stable even if its actual computation is unstable has been observed and studied thoroughly in kernel 
interpolation \cite{DeMarchiSchaback2008,DeMarchi2010}. Moreover, this discrepancy has been attributed to the use of the so-called direct method, i.e., the 
inversion of the kernel matrix, and several stable algorithm has been introduced to overcome this problem. 

In particular, given parameters $p\in\N$, $a\geq 0$ and a set of $\calh_{a,p}$-unisolvent points $X\subset \Omega$, in this section we show how to apply 
the RBF-QR algorithm to the polynomial kernels to construct a stable basis $\{u_j\}_{j=1}^N$ of $V_{a,p}(X)$. We follow Section 4.2 of \cite{Fornberg2011} and 
Section 4.1 in \cite{Fasshauer2012b}, where in our case the monomial basis plays the role of the Mercer basis of the kernel.

To this end we define an arbitrary but fixed ordering $\{\zeta^{(i)}\}_{i=1}^{M_a}$ of 
the multiindices $I_a(p,d)$ and construct the diagonal matrix $D\in\R^{M_a\times M_a}$ as 
in Lemma \ref{lemma:a_as_vdv}. Using the same ordering of the monomials, we split $D$ as 
\begin{equation}\label{eq:split_d}
D=\begin{bmatrix}D_1 &0 \\ 0 &D_2\end{bmatrix},
\end{equation}
now with $D_1\in \R^{N\times N}$ and $D_2\in \R^{(M_a-N)\times(M_a-N)}$. Observe in particular that the multi-indices may be sorted 
in such a way that the diagonal entries of $D$ are sorted in non-increasing order, so that $D_1$ contains the large entries, and $D_2$ the small ones that are 
likely the cause of the instability that we observed in the numerical inversion of $A$. Observe that, in case of equality among values of the diagonal, 
multiple possible ordering are possible. We are not investigating this aspect here, and just assume that an arbitrary valid order is fixed.

We then {assemble} the Vandermonde matrix $V\in \R^{N\times M_a}$ associated to the same monomial ordering and compute the QR decomposition $V=Q R$, with 
$Q\in\R^{N\times N}$ an orthogonal matrix, and where the matrix $R\in\R^{N\times M_a}$ is splitted as 
\begin{equation}\label{eq:split_r}
R:=[R_1|R_2], \;\; R_1\in \R^{N\times N},\;\; R_2\in \R^{N\times(M_a-N)}.
\end{equation}
It follows from Lemma \ref{lemma:a_as_vdv} that 
\begin{equation}\label{eq:A_V_C}
A 
= V D V^T
= V D R^T Q^T
= V \begin{bmatrix}D_1 &0 \\ 0 &D_2\end{bmatrix} \begin{bmatrix}R_1^T\\R_2^T\end{bmatrix} Q^T  
= V \begin{bmatrix}D_1 R_1^T \\ D_2 R_2^T\end{bmatrix} Q^T
= V \begin{bmatrix}I \\ D_2 R_2^T R_1^{-T} D_1^{-1}\end{bmatrix} D_1 R_1^T Q^T,
\end{equation}
where $D_1 R_1^T $ is invertible because $D$ is invertible and $V$, and thus $R$, have full rank since the points $X$ are unisolvent (see Proposition 
\ref{prop:unisolvent_points}).

We now recall that for an arbitrary basis $\{u_j\}_{j=1}^N$ of $V_{a,p}(X)$, if $C_u\in\R^{N\times N}$ is the matrix of 
change of basis from $\{k_{a,p}(\cdot, x_i)\}_{i=1}^N$ to this new basis, and $V_{u}:= \left[u_j(x_i)\right]_{i,j=1}^N$, then \cite{Pazouki2011} shows that it 
holds $A = V_u C_u^{-1}$. We may interpret the decomposition \eqref{eq:A_V_C} in these terms, and assume that $C_u^{-1}:=D_1 R_1^T Q^T$ is the inverse of the  
matrix of change of basis from the stable basis to the kernel basis of translates. With this definition, observe that we also have from \eqref{eq:A_V_C} that
\begin{align*}
V_u 
= A C_u 
= V \begin{bmatrix}I \\ D_2 R_2^T R_1^{-T} D_1^{-1}\end{bmatrix}
= V C_u',
\end{align*}
where now
\begin{equation}\label{eq:c_phi}
C_u':= \begin{bmatrix}I \\ D_2 R_2^T R_1^{-T} D_1^{-1}\end{bmatrix}\in\R^{M_a\times N}
\end{equation}
is a rectangular matrix that expresses the new basis in terms of the monomial basis. This can thus be used to 
express the stable basis without the need of passing through the unstable kernel basis. Observe moreover that $D_2 R_2^T R_1^{-T} 
D_1^{-1}$ is an $(M_a-N)\times N$ matrix, and that its computation (and thus the computation of $C_u$), requires only the inverses of the 
$N\times N$ matrices $R_1$ and $D_1$, which can be computed efficiently since $R_1$ is triangular and $D_1$ is diagonal.

We can thus work directly in terms of this new stable basis, both to solve the linear system and to evaluate the interpolant as
\begin{equation*}
I_{a,p,X}(x) = \sum_{j=1}^N (c_u)_j u_j(x), \;\;x\in\Omega,
\end{equation*}
for a suitable vector of coefficients $c_u\in\R^N$.
The computation of the 
interpolant is summarized in Algorithm \ref{alg:rbfqr_train}.
\begin{algorithm}[h!t]
  \caption{Construction of the interpolant}
  \begin{algorithmic}[1]
      \State{Input: Parameters $p\in\N$, $a\geq 0$, $\calh_{a,p}$-unisolvent interpolation points $X:=\{x_i\}_{i=1}^N\subset\R^d$ and target values $y\in\R^N$.}
      \State{Define an ordering $I_u:=\{\zeta^{(i)}\}_{i=1}^{M_a}$ of the multiindices $I_a(p,d)$.}
      \State{Evaluate the Vandermonde matrix $V\in\R^{N\times M_a}$ associated to the points $X$ and the given ordering (see Lemma \ref{lemma:a_as_vdv}).}
      \State{Compute the QR decomposition $Q R = V$, and define $R_1,R_2$ as in \eqref{eq:split_r}.}
      \State{Evaluate the diagonal matrix $D\in\R^{M_a\times M_a}$ (see Lemma \ref{lemma:a_as_vdv}) and define $D_1,D_2$ as in \eqref{eq:split_d}.}
      \State{Compute $C_u'$ as in \eqref{eq:c_phi}.}
      \State{Compute the coefficients $c_u:= C_u^{-1} y$.}
     \State{Output: $C_u'$, $c_u$, $I_u$.}
    \end{algorithmic}\label{alg:rbfqr_train}
\end{algorithm}

Once the matrix $C_u'$, the vector $c_u$, and the ordering $I_u$ are computed, they can be used to evaluate the interpolant on any set 
$X_{eval}:=\{x'_i\}_{i=1}^{N_{eval}}\subset \R^d$ of $N_{eval}\in\N$ evaluation points, again by using the corresponding Vandermonde matrix. We describe this 
process in  Algorithm \ref{alg:rbfqr_test}, which returns the vector $y_{eval}$ with $(y_{eval})_i:=I_{a,p,X}(x'_i)$.

\begin{algorithm}[ht!]
  \caption{Evaluation of the interpolant}
  \begin{algorithmic}[1]
      \State{Input: Evaluation points $X_{eval}:=\{x'_i\}_{i=1}^{N_{eval}}\subset\R^d$, $N_{eval}\in\N$, matrix $C_u'$, vector $c_u$, and indices $I_u$ from 
Algorithm 
\ref{alg:rbfqr_train}.}
      \State{Evaluate the Vandermonde matrix $V\in\R^{N_{eval}\times M_a}$ (see Lemma \ref{lemma:a_as_vdv}) associated to the points $X_{eval}$ and the 
ordering $I_u$.}
      \State{Evaluate the stable basis on $X_{eval}$ as $V_u:= V C_u'$.}
      \State{Evaluate the interpolant {on $X_{eval}$} as $y_{eval} = V_u c_u$.}
     \State{Output: $y_{eval}$.}
    \end{algorithmic}\label{alg:rbfqr_test}
\end{algorithm}

We remark that this approach works exactly in the same way to interpolate vector-valued functions $f:\R^d\to \R^{d'}$, $d'\in\N$. In this case $y\in 
\R^{N\times d'}$ collects the evaluations of $f$ on $X$ rowwise, and the output $c_u$ of Algorithm \ref{alg:rbfqr_train} is a matrix $c_u\in \R^{N\times d'}$. 
In particular, this approach can also be used to compute the Lagrange basis functions \eqref{eq:cardinal_interpolant} by a single run of Algorithm 
\ref{alg:rbfqr_train}, simply by defining $d':=N$ and $y:=I\in \R^{N\times N}$. The resulting output $y_{eval}\in \R^{N_{eval}\times N}$ of Algorithm 
\ref{alg:rbfqr_test} contains as columns the Lagrange functions evaluated on $X_{eval}$ (see Remark 13.8 in \cite{Fasshauer2015}).

\begin{remark}\label{rem:rbf_qr_better}
The RBF-QR algorithm includes also a series of further ad-hoc optimizations that we are not {discussing} here for simplicity. 
Moreover, {the stable basis} obtained by this process is expressed in terms of the monomial basis, and it is thus clearly a polynomial basis. 
Similar approaches based on QR decompositions are used for point selection in polynomial interpolation, such as the Approximate Fekete Points (AFP) of 
\cite{Bos2010}. In particular, the matrices $Q$ and $V_u$ should be related to some sort of orthogonal polynomials, which could be interesting to investigate 
to 
further improve the stability of the algorithm.
\end{remark}

\section{Numerical experiments}\label{sec:numerics}
We test now numerically some aspects that were discussed in the previous sections. For simplicity we restrict to $d=1$ and set $\Omega=[-1, 1]$, so 
that \eqref{eq:index_sets} and \eqref{eq:index_sets_dim} give $I_a(p,1) = \{0, 1, \dots, p\}$ and $M_a = p + 1$ if $a>0$, while $I_0(p,1) = \{p\}$ 
and $M_0 = 1$. Moreover \eqref{eq:expansion_kernel} and \eqref{eq:expansion_kernel_hom} simplify to
\begin{equation*}
k_{a,p}(x, y) = \sum_{\zeta=0}^{p} \binom{p}{\zeta} a^{p-\zeta} (xy)^\zeta,
\;\;\;\;
k_{0,p}(x, y) = (xy)^p.
\end{equation*}
Some examples of the values of $k_{a,p}(\cdot, 1/2)$ are visualized in Figure~\ref{fig:kernel_viz}.

\begin{figure}[h!t]
\begin{center}
\includegraphics[width=\textwidth]{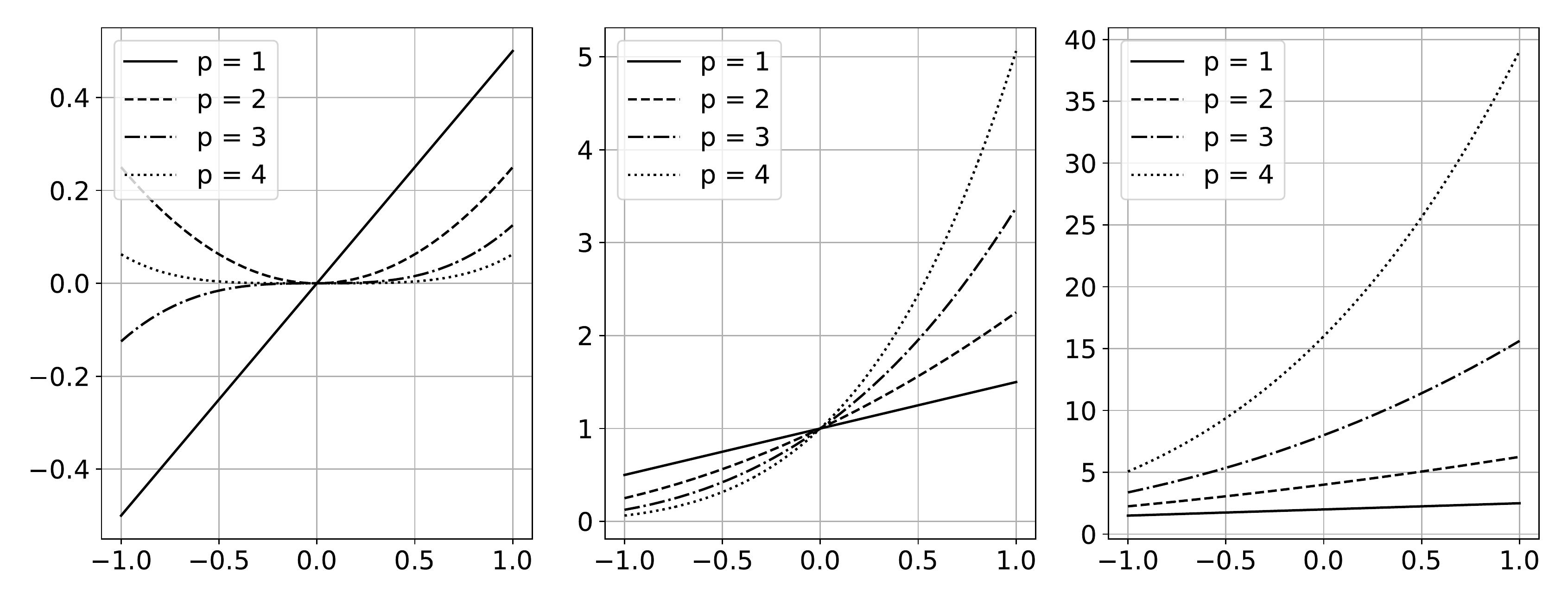}
\end{center}

\caption{Values of the kernel $k_{a,p}(\cdot, 1/2)$ on $[-1,1]$ for $p\in\{1,2,3,4\}$ and for $a=0$ (left), $a=1$ (center), and $a=2$ (right).}
\label{fig:kernel_viz}
\end{figure}

\subsection{Convergence of the interpolant and stable computations}\label{sec:numerics_conv} 
We start by comparing kernel interpolation with polynomial interpolation, and demonstrate the potential instability of the direct method and the 
effectiveness of the RBF-QR approach. 

We recall that $M_a=p+1$ for $d=1$, and thus the constraint $N\leq M_a$ means that we need to require $p\geq N-1$. Moreover, for any $p\in\N$ any set of 
pairwise distinct points is the subset of a set of $\pol_p^1$-unisolvent points, and thus in view of Proposition \ref{prop:unisolvent_points} we can solve 
interpolation problems with any set of $N$ pairwise distinct points provided that $p\geq N-1$.

As an example we interpolate the smooth function $f(x) := \cos(10 x)$ sampled at $N$ Chebyshev points with $N=5, \dots, 50$. For each $N$ we test polynomial 
interpolation, and kernel interpolation with $p=p(N)\in\{N-1, N+1, N+3, N+5\}$ and $a\in\{5, 10\}$. 

We report in Figure \ref{fig:1d_convergence} the corresponding maximum absolute errors with respect to the exact values of $f$, computed on a grid of 
$N_{eval}=1000$ equally spaced points. It is remarkable to observe that 
the direct approach (first row of Figure \ref{fig:1d_convergence}) fails to compute a converging interpolant even for small values of $N$ and even if the 
corresponding polynomial interpolant is stable and convergent. On the other hand, switching to RBF-QR (second row of Figure \ref{fig:1d_convergence}) resolves 
this instability, and for all the tested values 
of $p$ and $a$ the kernel interpolants converge at the same speed of the polynomial interpolant.
However, also in this stable case the convergence saturates at an error of roughly $10^{-12}$ at $N=30$, and from there on the stability seems to slightly 
degrade, up to some 
small oscillations for large $N$. It is possible that the adoption of further optimizations, as discussed in Remark \ref{rem:rbf_qr_better}, 
may lead to an even stronger stability.

\begin{figure}
\begin{center}
\begin{tabular}{cc}
\includegraphics[width=.45\textwidth]{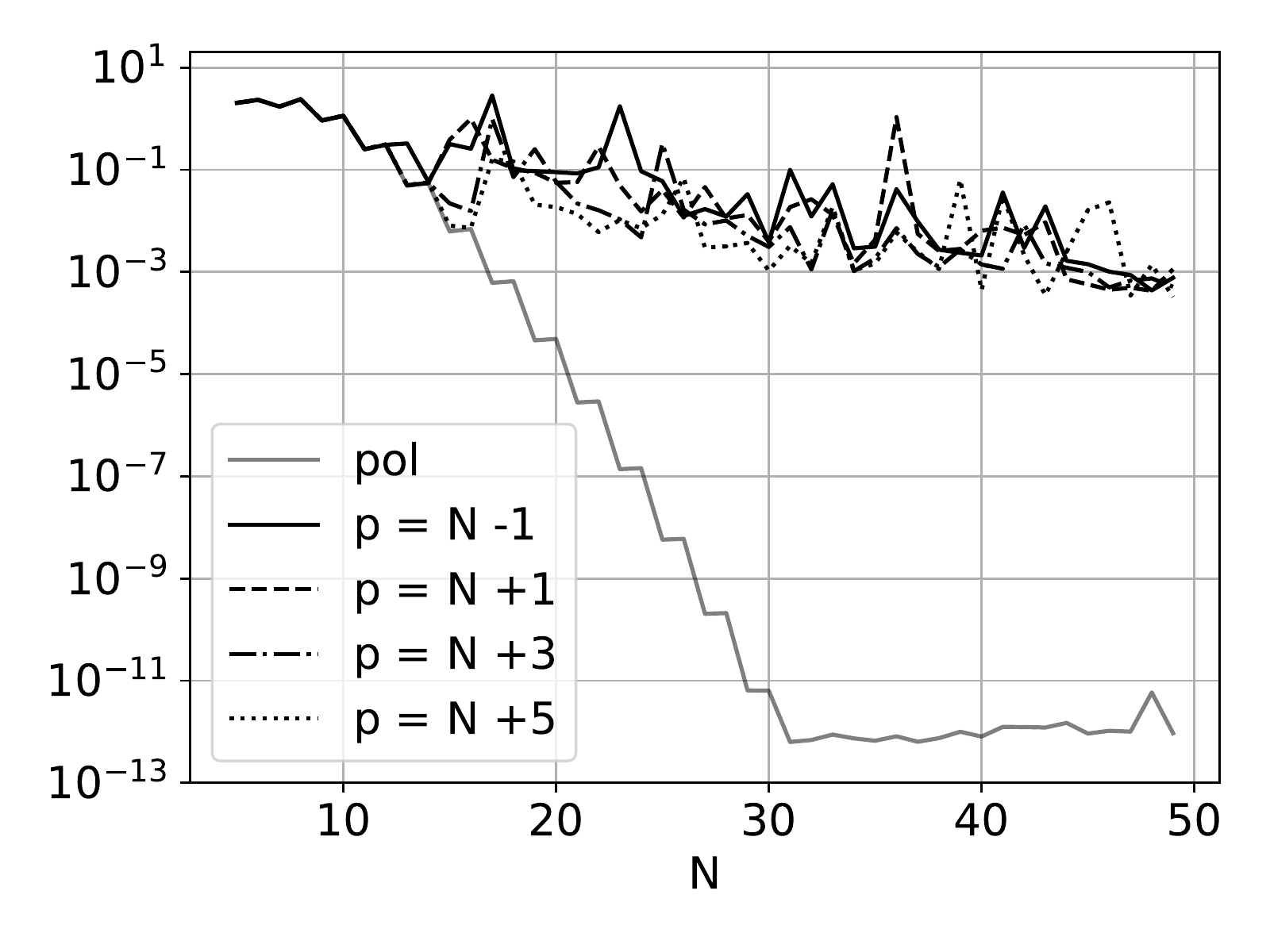}
&\includegraphics[width=.45\textwidth]{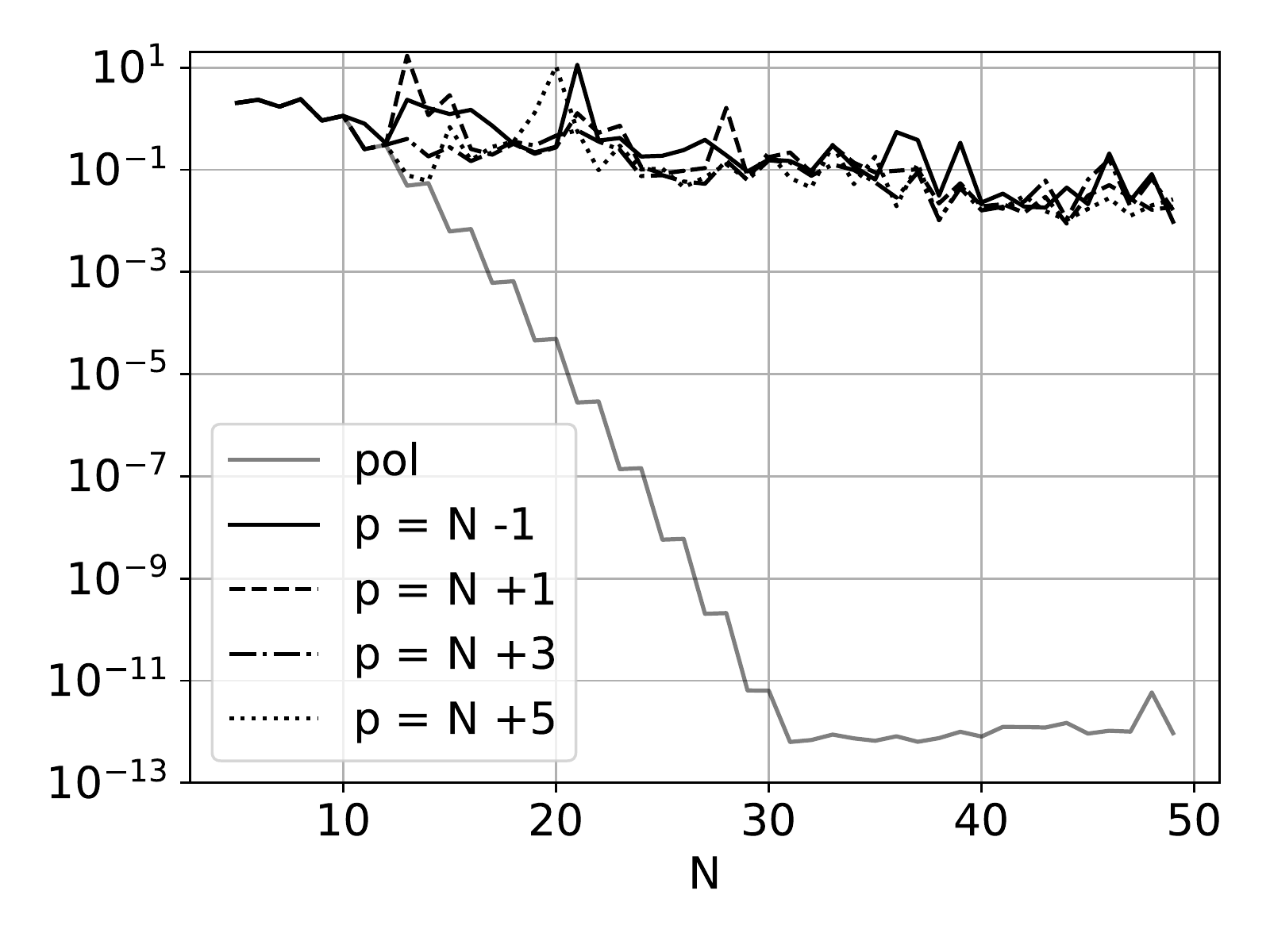}
\\
\includegraphics[width=.45\textwidth]{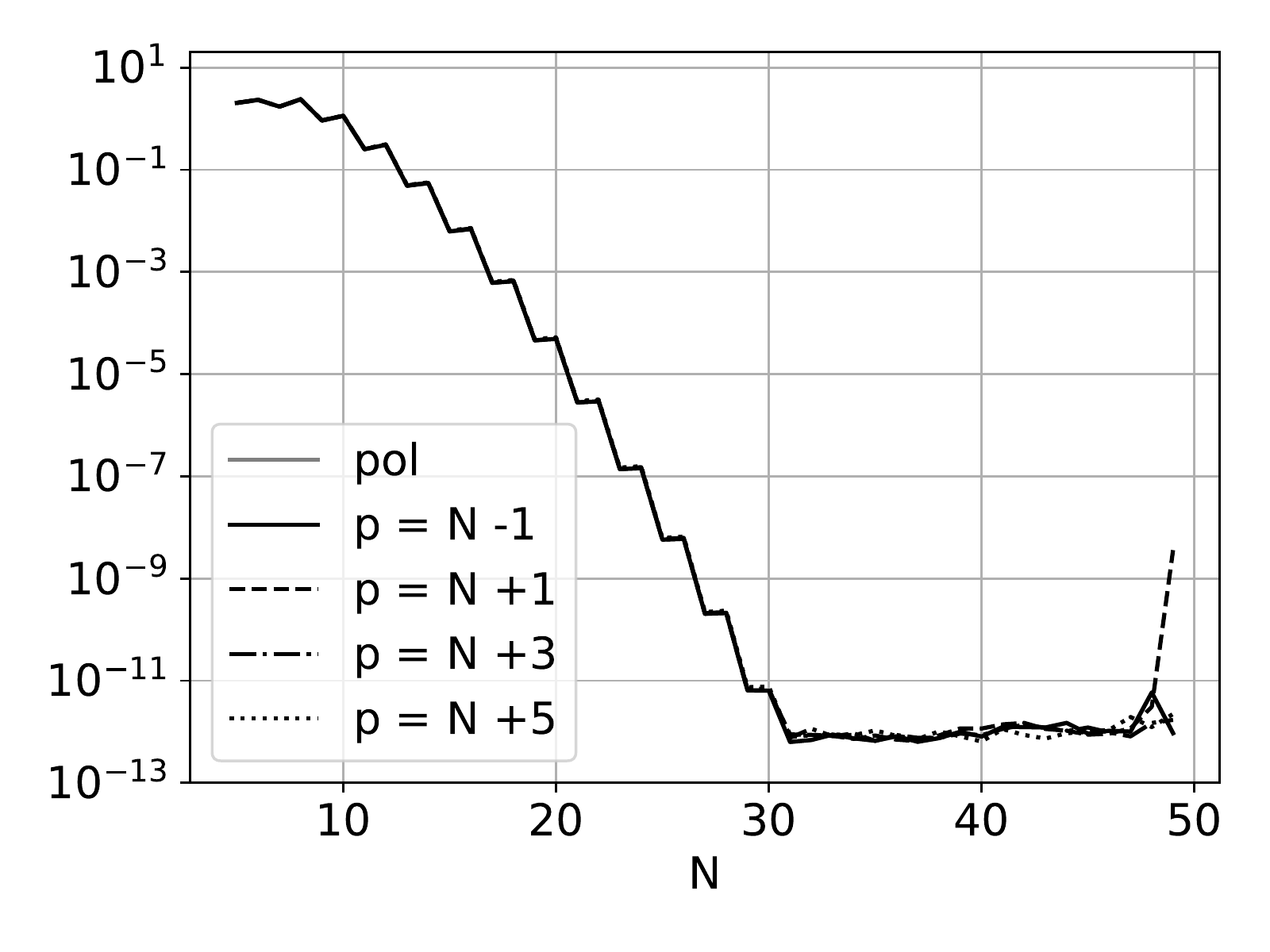}
&\includegraphics[width=.45\textwidth]{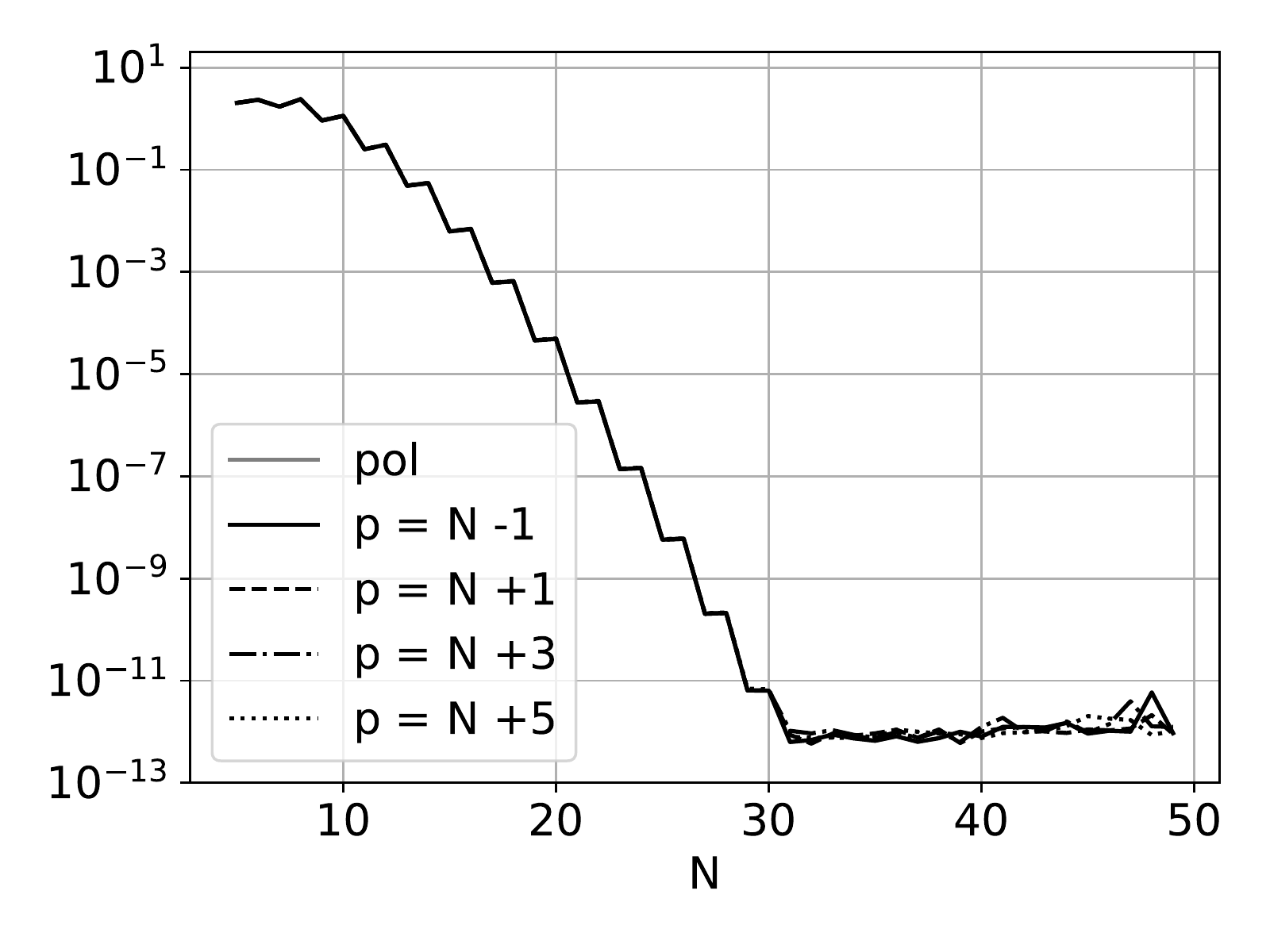}
\end{tabular}
\end{center}

\caption{Convergence of the maximal absolute error of interpolation of the function $f(x) = \cos(10 x)$ using $N=5, \dots, 50$ Chebyshev points. For each 
figure, we test a polynomial interpolant (gray line), and kernel interpolants with various values of $p$, and $a=5$ (left column) and $a=10$ (right column). 
The kernel interpolants are computed with the direct method (first row) and with RBF-QR (second row).}
\label{fig:1d_convergence}
\end{figure}

\subsection{Lagrange functions and Lebesgue constant}
We now consider the stability aspects of the interpolant, analysing the associated Lagrange functions and Lebesgue constant. 

We first show that RBF-QR is indeed effective also in computing these cardinal functions. As an example, Figure \ref{fig:lagrange_qr} shows the Lagrange 
functions of $k_{10, 25}$ corresponding to $N=15$ Chebyshev points. Even for this small number of points, it is immediately clear that a stable algorithm is 
needed to have an accurate computation.
\begin{figure}[h!t]
\begin{center}
\includegraphics[width=\textwidth]{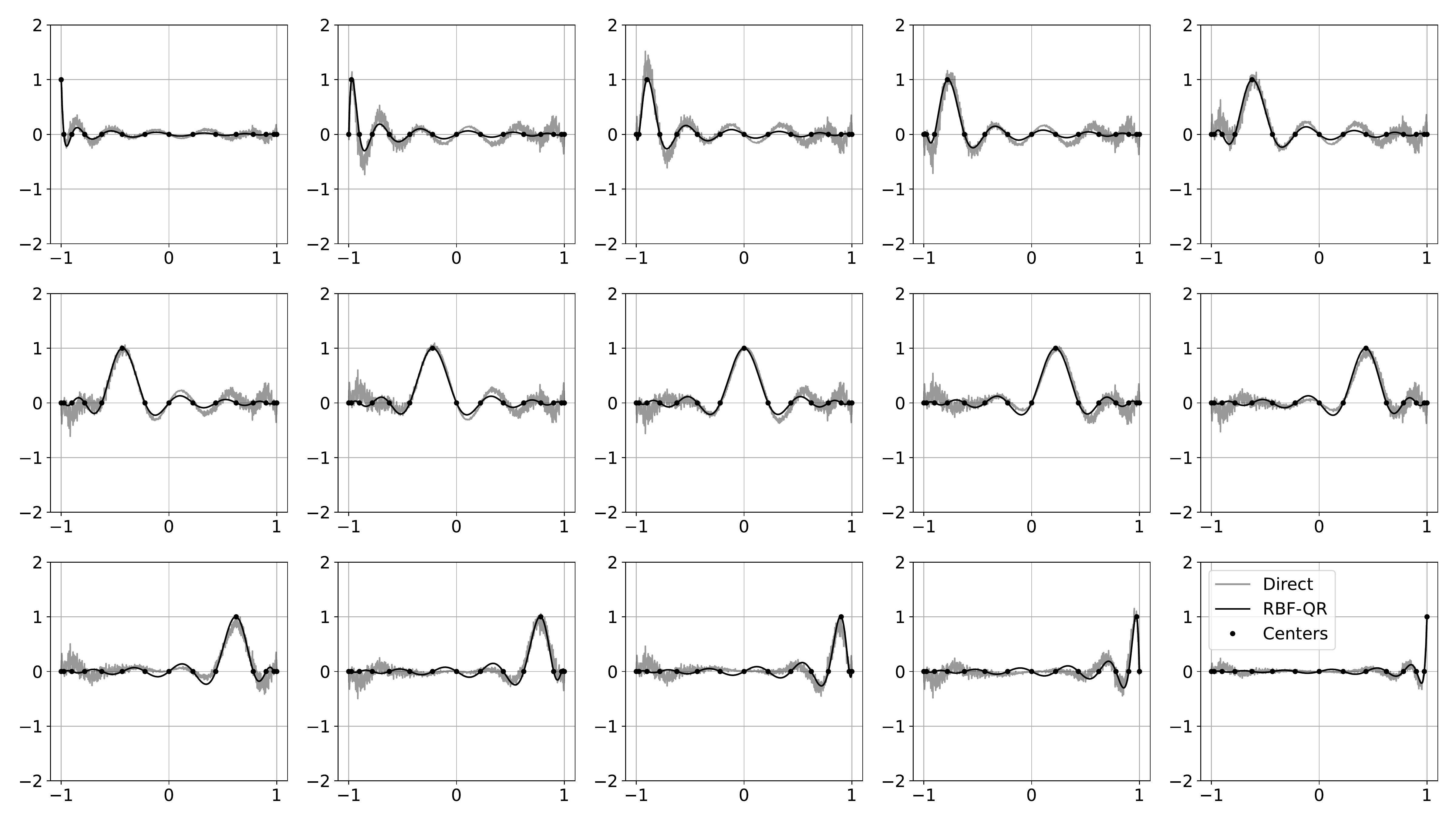}
\end{center}

\caption{Lagrange functions for the polynomial kernel with $p=25$ and $a=10$, and corresponding to $N=15$ Chebyshev points (black dots), computed 
with the direct method (gray lines) and with RBF-QR (black lines).}
\label{fig:lagrange_qr}
\end{figure}

We use these stably computed functions to evaluate the Lebesgue function and the Lebesgue constant and compare it with the ones of polynomial interpolation. 
As an example, for $N=5$ Chebyshev points we consider the polynomial kernel $k_{a,p}$ with $p\in\{N-1, N+9, N+19, N+29\}$ and $a=5$. The first row in Figure 
\ref{fig:lagr_and_leb} shows the Lagrange functions for the different values of $p$, and it is clear that an increase of $p$ has the effect of smoothing the 
oscillations in the interior of the domain and enlarging the oscillations close to the boundary. This reflects into a Lebesgue function (left panel in Figure 
\ref{fig:lagr_and_leb}) which is decreasing in the interior and increasing close to the boundary as $p$ increases. This behavior causes the Lebesgue function 
(right panel in Figure \ref{fig:lagr_and_leb}) to be initially decreasing and then increasing. In particular, the minimum value for this set of points and 
parameter $a$ is reached for $p>N-1$, i.e., there exists a polynomial kernel with a Lebesgue constant which is strictly smaller than that of the polynomial 
interpolant.

\begin{figure}[h!t]
\begin{center}
\begin{tabular}{cc}
\multicolumn{2}{c}{\includegraphics[width=.92\textwidth]{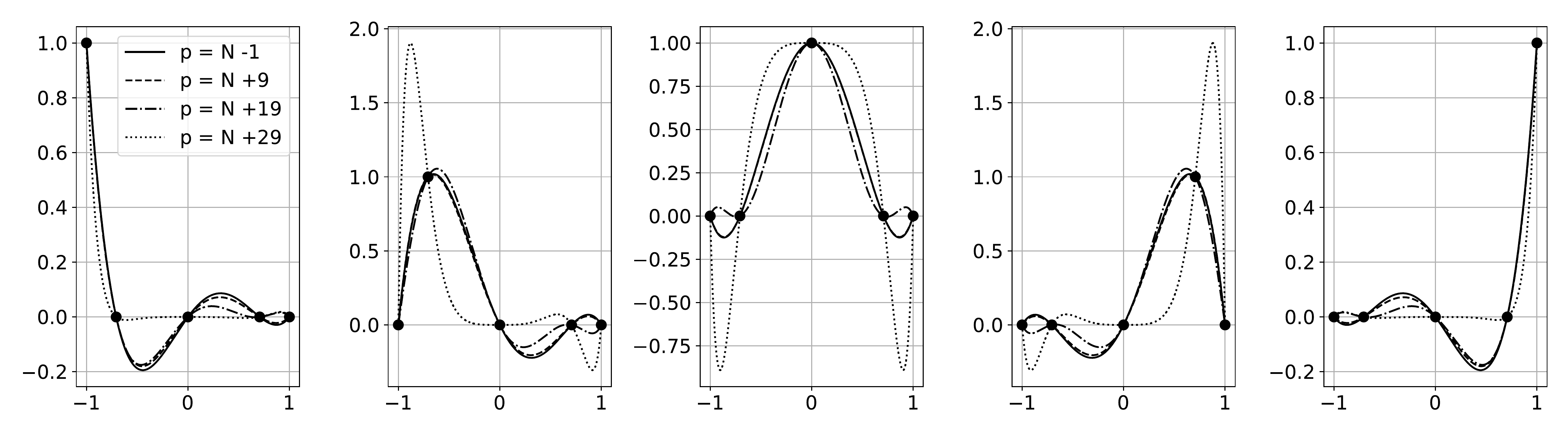}}
\\
\includegraphics[width=.45\textwidth]{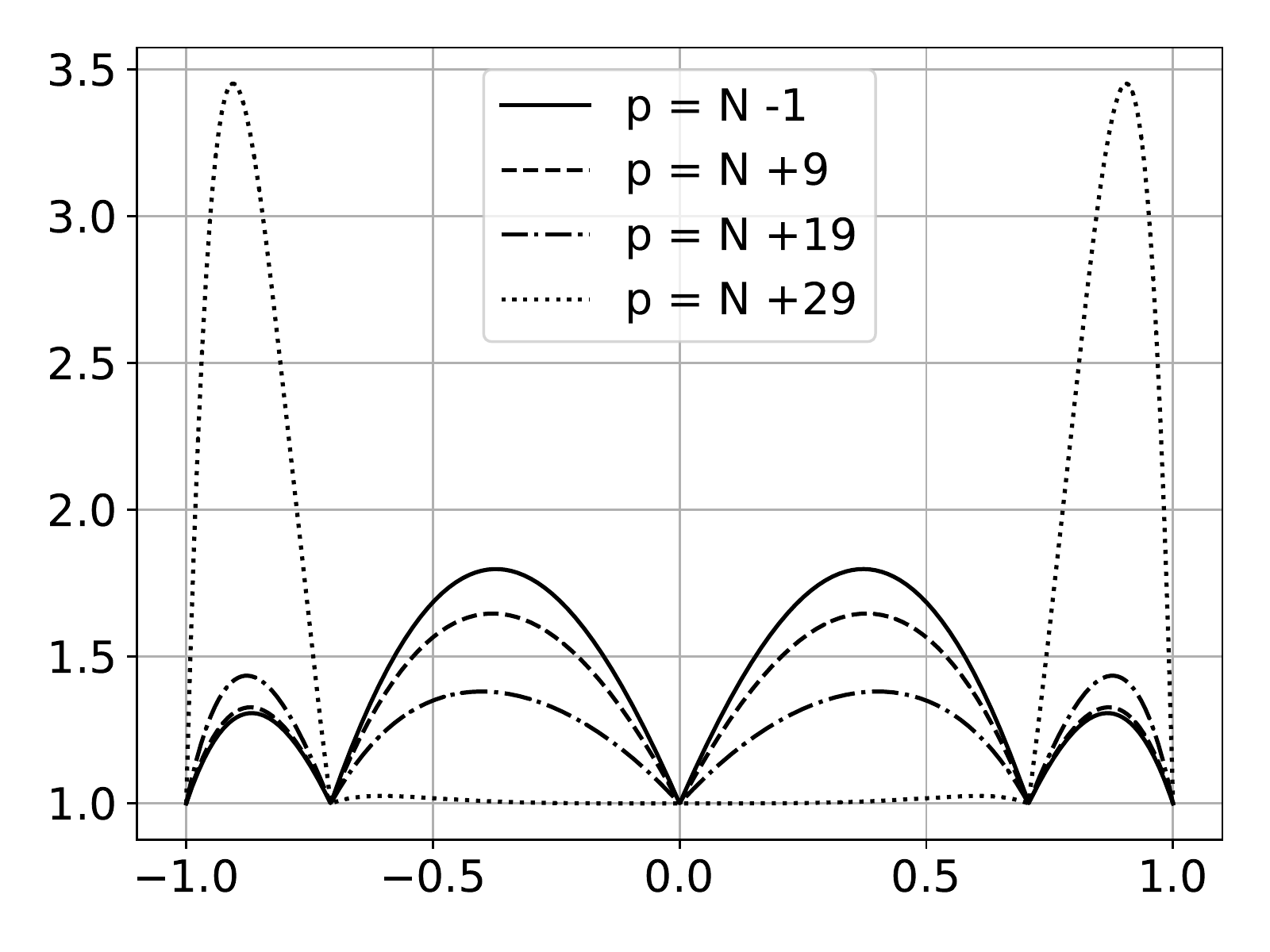}
&\includegraphics[width=.45\textwidth]{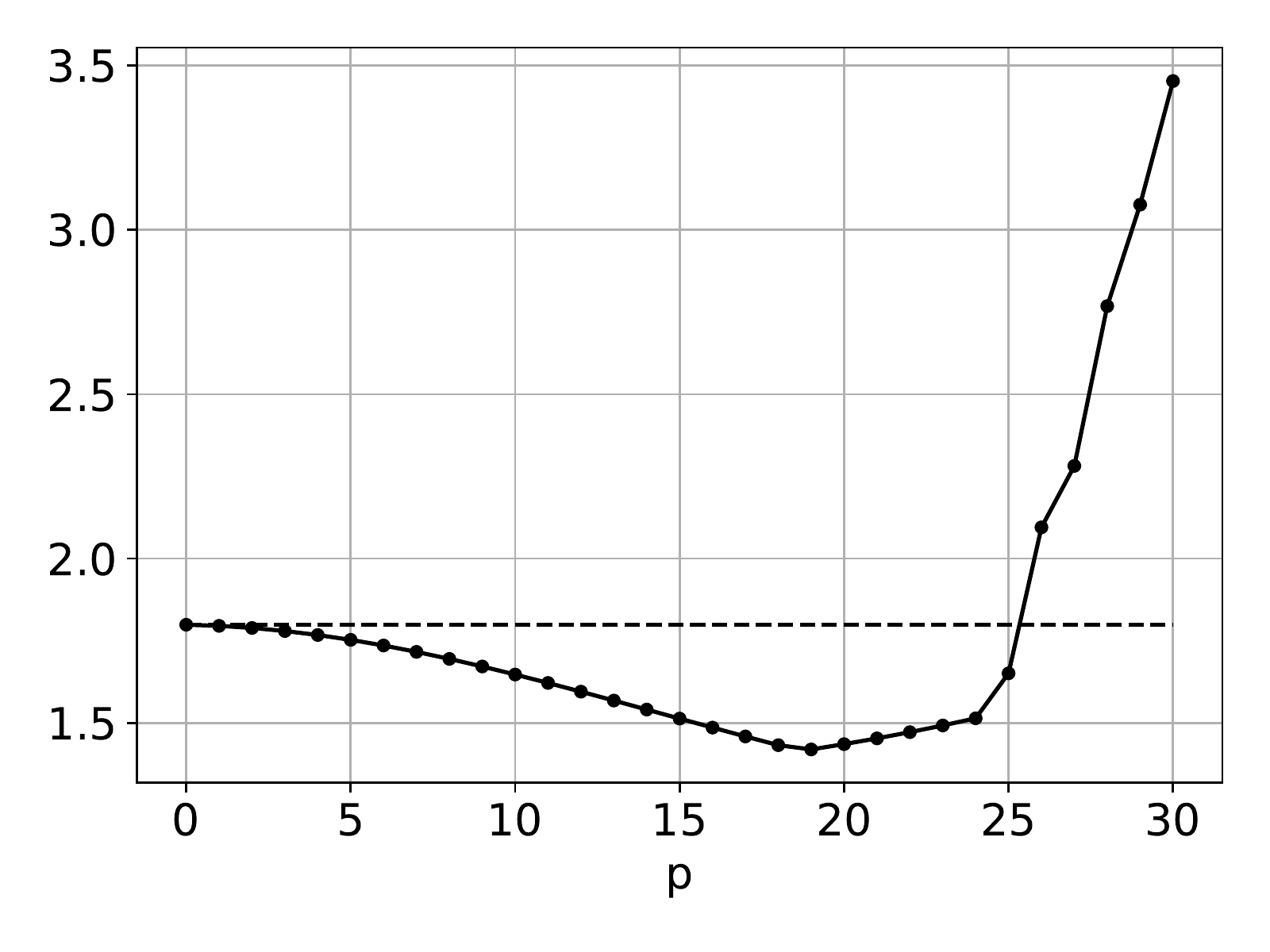}
\end{tabular}
\end{center}

\caption{Lagrange functions (top), Lebesgue function (bottom left), and Lebesgue constant (bottom right) for interpolation with a polynomial kernel $k_{5, p}$ 
on $N=5$ Chebyshev points and various values of $p$, as reported in each panel.
}
\label{fig:lagr_and_leb}
\end{figure}

To further investigate the stability of the interpolants we look at the asymptotic behavior of the Lebesgue constant. We compare the same kernels $k_{a,p}$ 
used 
in the previous section, but now testing both equally spaced and Chebyshev points, for $N=5, \dots, 45$. We restrict to a maximal 
$N=45$ because the same 
instability in the computations observed in Section \ref{sec:numerics_conv} appears here for large $N$, up to making the results completely unreliable for 
$N\approx 50$. The results are reported in 
Figure \ref{fig:1d_lebesgue}. It is clear that in all cases the growth of the Lebesgue constant coincides with that of polynomial interpolation (i.e., 
$p=N-1$), 
and thus it has the well known logarithmic growth for Chebyshev points (left panels in Figure \ref{fig:1d_lebesgue}), and exponential growth for equally spaced 
points (right panels in Figure \ref{fig:1d_lebesgue}). It is important to notice that the growth seems to be not affected by the value of $a$, and especially 
not even by that of $p$. This latter fact is relevant because it implies that, at least for $d=1$, the value of $\Lambda_{X, a, p}$ depends on $X$ and not on 
$X_{M_a}$, and especially it seems that the bound of Theorem \ref{thm:stability} is quite pessimistic. Moreover, it seems that the difference between the 
values of the Lebesgue constants of polynomial and kernel interpolation observed in Figure \ref{fig:lagr_and_leb} (bottom right) are not so significant for 
growing $N$.

\begin{figure}[h!t]
\begin{center}
\begin{tabular}{cc}
\includegraphics[width=.45\textwidth]{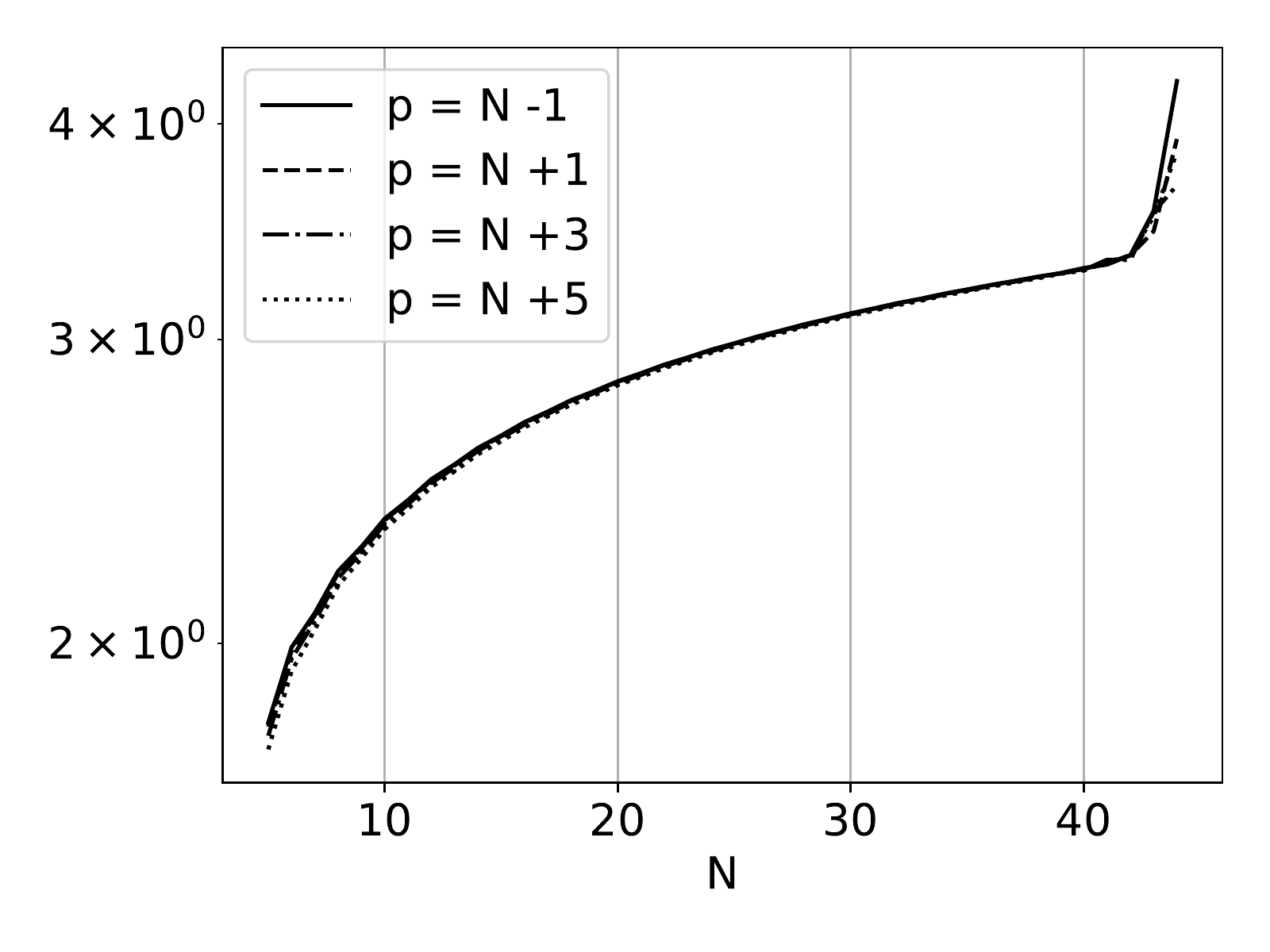}
&\includegraphics[width=.45\textwidth]{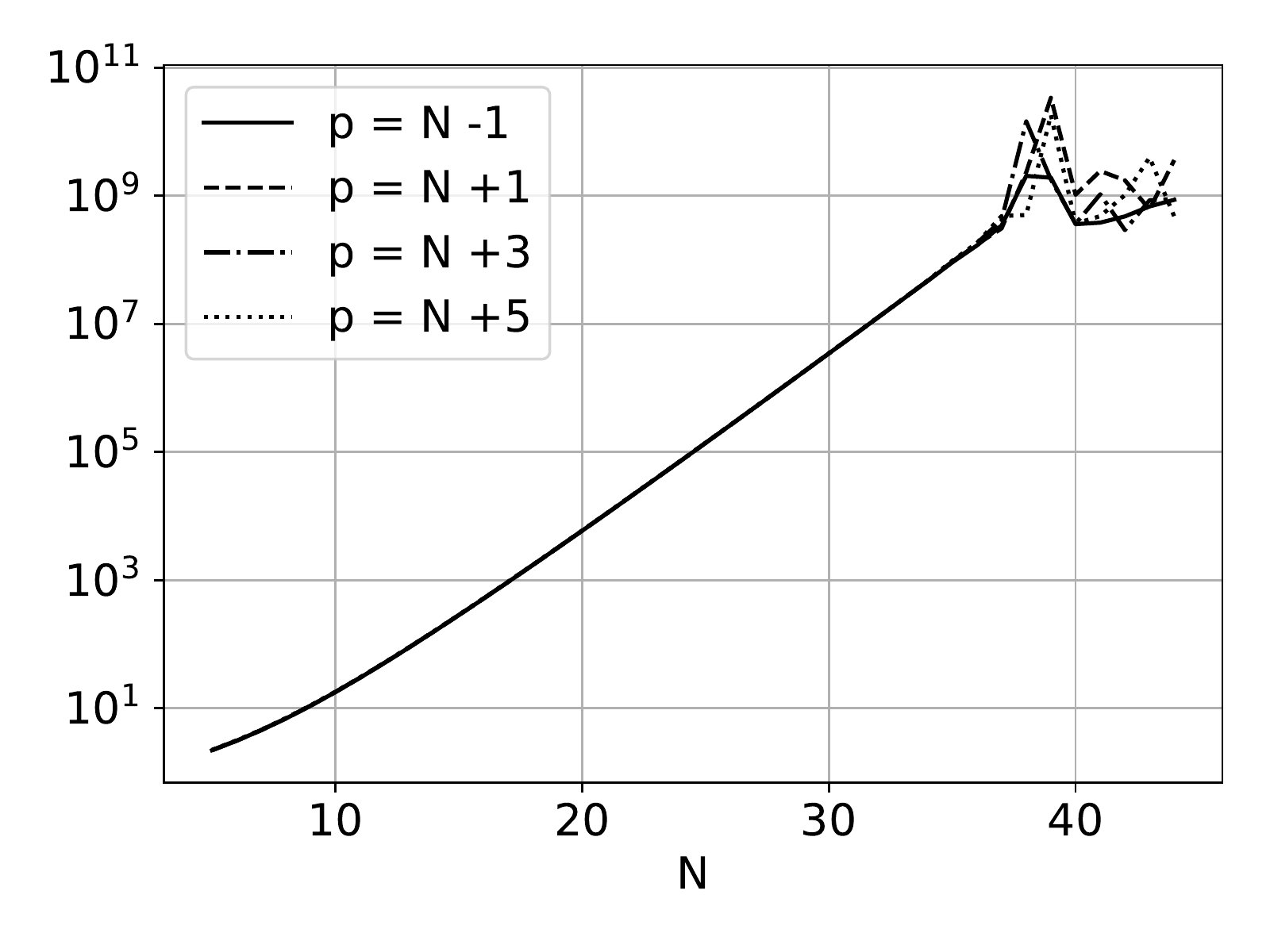}
\\
\includegraphics[width=.45\textwidth]{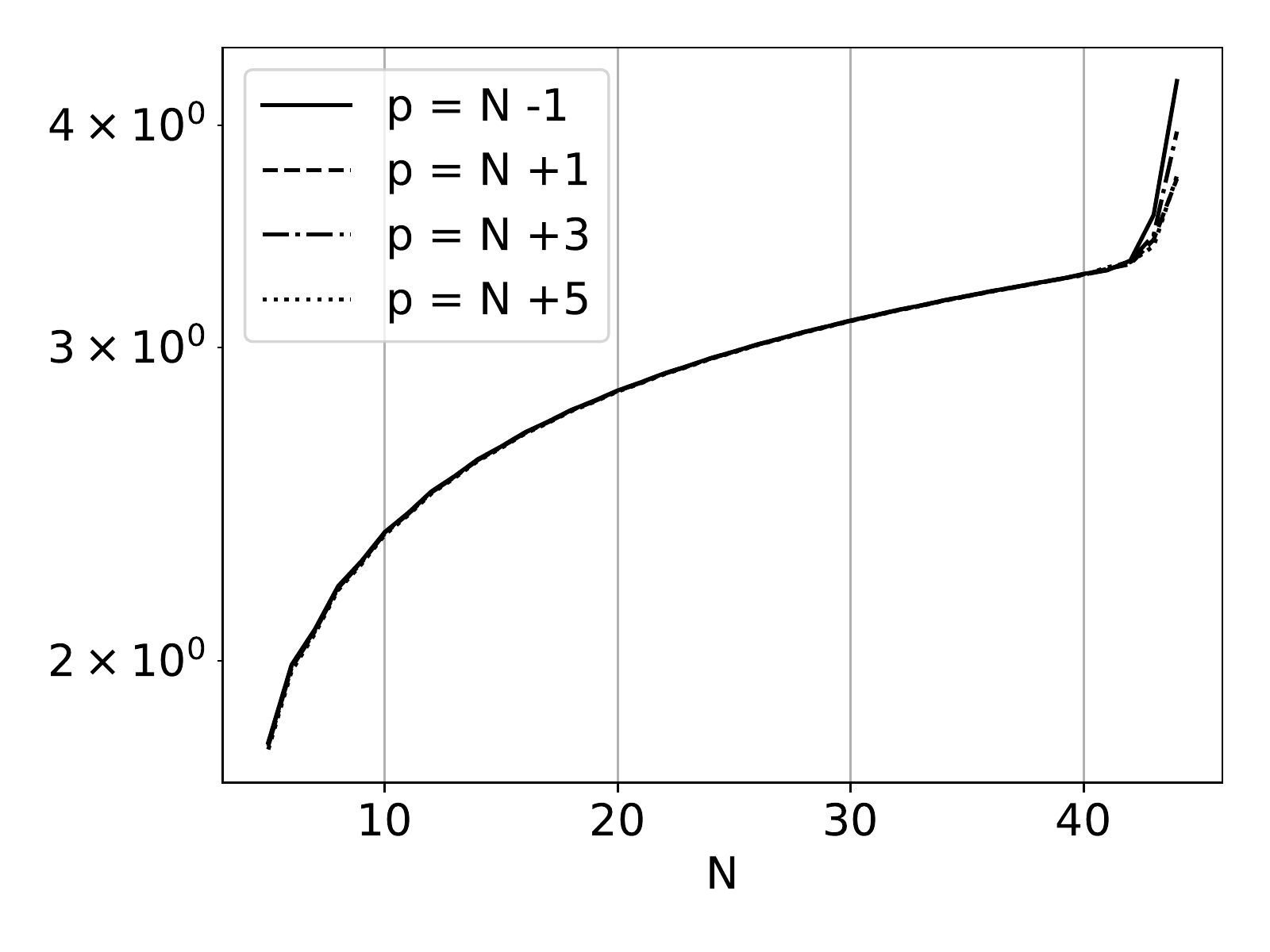}
&\includegraphics[width=.45\textwidth]{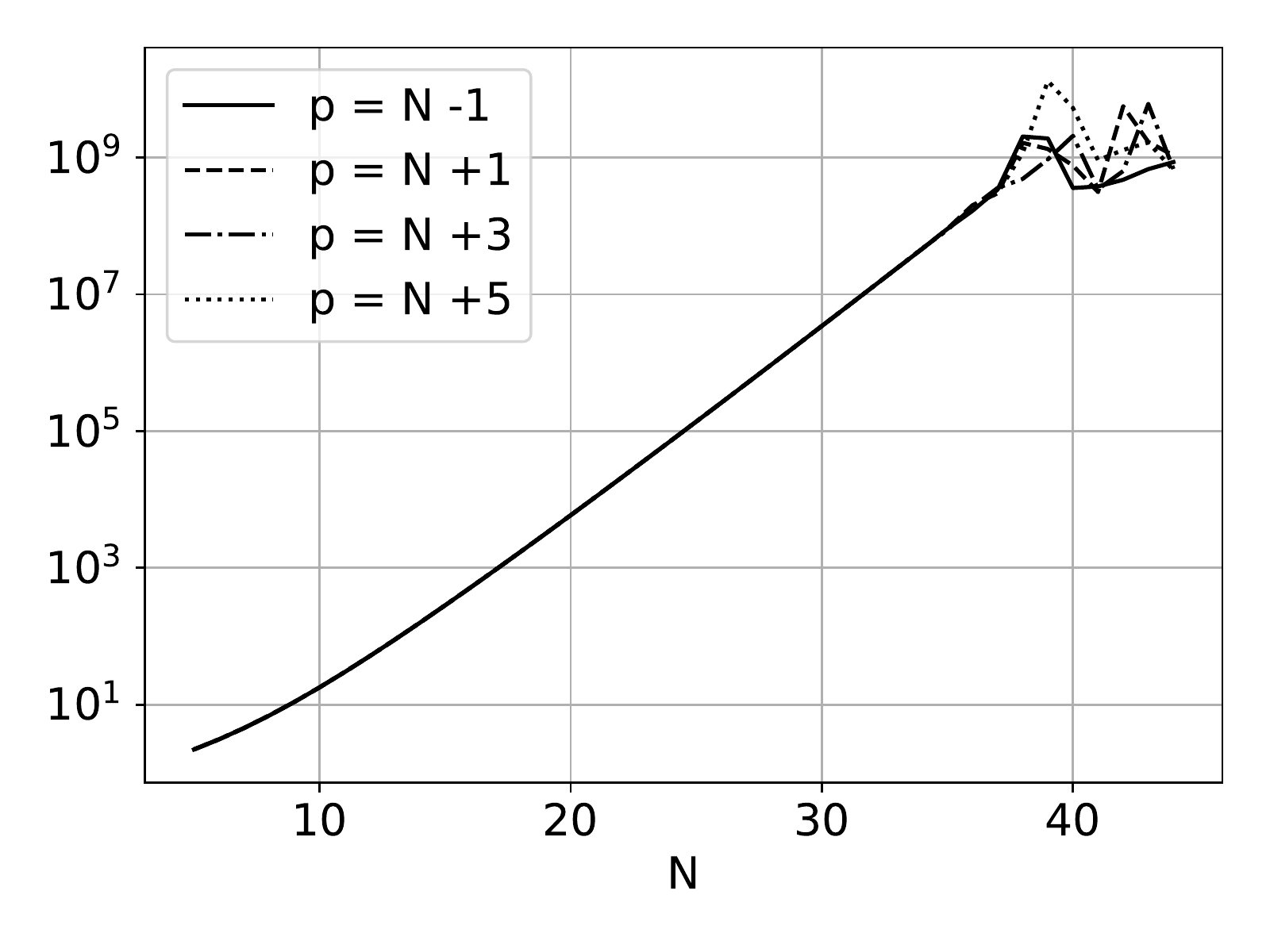}
\end{tabular}
\end{center}

\caption{Growth of the Lebesgue constant associated to $N=5, \dots, 45$ Chebyshev (left) and equally spaced (right) points. We test kernel interpolants 
with various values of $p$, and $a=5$ (top row) and $a=10$ (bottom row). We remark that the results for $p=N-1$ coincide with those of polynomial 
interpolation.}
\label{fig:1d_lebesgue}
\end{figure}

\section{Conclusions and perspectives}\label{sec:conclusions}
In this paper we derived some initial results for the application and analysis of polynomial kernels for the solution of interpolation problems. 
We derived necessary and sufficient conditions for the existence of a unique interpolant and we provided an explicit description of the native spaces of the 
polynomial kernels. In particular, we analyzed in some detail the effect of the kernel parameters on these spaces.
These results were further used to derive some first quantification of the stability and convergence of these interpolants, with particular attention to 
the connection with the corresponding results in polynomial interpolation. Finally, we have shown that a direct solution of the interpolation system leads to 
inaccurate computations, and that the use of the RBF-QR algorithm can significantly mitigate this issue.

Several points remain open and will be the subject of future research. In particular, we outlined in several occasions that a proper selection of the degree 
$p$ 
may be crucial, and that its choice may balance between accuracy and stability. A better understanding of the role of this parameter and a systematic method 
for its determination are open problems. This aspect may be connected to the so-called overparameterized regime and to ridgeless regression in machine learning, 
since by increasing $p$ one may aim at solving a data fitting problem by interpolation without regularization, and use the parameter $a$ as an implicit 
regularizer (see e.g. \cite{Liang2020,Pagliana2020,Richards2021}). 

Moreover, some results of this paper point to the fact that the properties of an interpolation set $X$ may be related to those of a superset $X_{M_a}$ of 
polynomially unisolvent points. Also in this case a quantitative relation is missing, as well as suitable algorithms to select $X$ from $X_{M_a}$ or completing 
$X$ to $X_{M_a}$. In both cases, it would be interesting to investigate processes related to Leja and approximate Fekete points in this context, as well 
as to $P$-greedy points \cite{DeMarchi2005}. 

\paragraph{Acknowledgements}
This research has been accomplished within the Rete ITaliana di Approssimazione (RITA) and the thematic group on Approximation Theory and Applications of the Italian Mathematical Union (UMI).
The authors would like to thank the anonymous reviewers for providing several comments that helped improving this paper.

\bibliography{biblio.bib}

\begin{thebibliography}{10}

\bibitem{Belkin2018}
M.~Belkin.
\newblock Approximation beats concentration? {A}n approximation view on
  inference with smooth radial kernels.
\newblock In {\em Conference On Learning Theory, {COLT} 2018, Stockholm,
  Sweden, 6-9 July 2018.}, pages 1348--1361, 2018.

\bibitem{Bialas2016}
L.~Bia{\l}as-Cie{\.z} and J.-P. Calvi.
\newblock Homogeneous minimal polynomials with prescribed interpolation
  conditions.
\newblock {\em Transactions of the American Mathematical Society},
  368(12):8383--8402, 2016.

\bibitem{Bos2006}
L.~Bos, M.~Caliari, S.~{De Marchi}, M.~Vianello, and Y.~Xu.
\newblock Bivariate {L}agrange interpolation at the {P}adua points: The
  generating curve approach.
\newblock {\em Journal of Approximation Theory}, 143(1):15--25, 2006.
\newblock Special Issue on Foundations of Computational Mathematics.

\bibitem{Bos2010}
L.~Bos, S.~De~Marchi, A.~Sommariva, and M.~Vianello.
\newblock Computing multivariate {F}ekete and {L}eja points by numerical linear
  algebra.
\newblock {\em SIAM Journal on Numerical Analysis}, 48(5):1984--1999, 2010.

\bibitem{Bos2011b}
L.~Bos, S.~De~Marchi, A.~Sommariva, and M.~Vianello.
\newblock Weakly admissible meshes and discrete extremal sets.
\newblock {\em Numerical Mathematics: Theory, Methods and Applications},
  4(1):1--12, 2011.

\bibitem{Bos2007}
L.~Bos, S.~De~Marchi, M.~Vianello, and Y.~Xu.
\newblock Bivariate {L}agrange interpolation at the {P}adua points: the ideal
  theory approach.
\newblock {\em Numerische Mathematik}, 108(1):43--57, Nov 2007.

\bibitem{Buhmann2003}
M.~D. Buhmann.
\newblock {\em Radial {B}asis {F}unctions: theory and implementations},
  volume~12 of {\em Cambridge Monographs on Applied and Computational
  Mathematics}.
\newblock Cambridge University Press, Cambridge, 2003.

\bibitem{sklearn_api}
L.~Buitinck, G.~Louppe, M.~Blondel, F.~Pedregosa, A.~Mueller, O.~Grisel,
  V.~Niculae, P.~Prettenhofer, A.~Gramfort, J.~Grobler, R.~Layton,
  J.~VanderPlas, A.~Joly, B.~Holt, and G.~Varoquaux.
\newblock {API} design for machine learning software: experiences from the
  scikit-learn project.
\newblock In {\em ECML PKDD Workshop: Languages for Data Mining and Machine
  Learning}, pages 108--122, 2013.

\bibitem{Caliari2005}
M.~Caliari, S.~{De Marchi}, and M.~Vianello.
\newblock Bivariate polynomial interpolation on the square at new nodal sets.
\newblock {\em Applied Mathematics and Computation}, 165(2):261--274, 2005.

\bibitem{Caliari2008}
M.~Caliari, S.~Marchi, and M.~Vianello.
\newblock Algorithm 886: {P}adua2d---{L}agrange interpolation at {P}adua points
  on bivariate domains.
\newblock {\em ACM Trans. Math. Softw.}, 35(3), Oct 2008.

\bibitem{Chung1977}
K.~C. Chung and T.~H. Yao.
\newblock On lattices admitting unique {L}agrange interpolations.
\newblock {\em SIAM J. Numer. Anal.}, 14(4):735--743, 1977.

\bibitem{DeMarchiSchaback2008}
S.~De~Marchi and R.~Schaback.
\newblock Stability constants for kernel-based interpolation processes.
\newblock Technical report, Dipartimento di Informatica, Universit{\`a} degli
  Studi di Verona, 2008.

\bibitem{DeMarchi2010}
S.~De~Marchi and R.~Schaback.
\newblock Stability of kernel-based interpolation.
\newblock {\em Adv. Comput. Math.}, 32(2):155--161, 2010.

\bibitem{DeMarchi2005}
S.~De~Marchi, R.~Schaback, and H.~Wendland.
\newblock Near-optimal data-independent point locations for {R}adial {B}asis
  {F}unction interpolation.
\newblock {\em Adv. Comput. Math.}, 23(3):317--330, 2005.

\bibitem{Fasshauer2007}
G.~E. Fasshauer.
\newblock {\em Meshfree {A}pproximation {M}ethods with {MATLAB}}, volume~6 of
  {\em Interdisciplinary Mathematical Sciences}.
\newblock World Scientific Publishing Co. Pte. Ltd., Hackensack, {NJ}, 2007.

\bibitem{Fasshauer2015}
G.~E. Fasshauer and M.~McCourt.
\newblock {\em Kernel-{B}ased {A}pproximation {M}ethods Using {MATLAB}},
  volume~19 of {\em Interdisciplinary Mathematical Sciences}.
\newblock World Scientific Publishing Co. Pte. Ltd., Hackensack, {NJ}, 2015.

\bibitem{Fasshauer2012b}
G.~E. Fasshauer and M.~J. McCourt.
\newblock Stable evaluation of {G}aussian {R}adial {B}asis {F}unction
  interpolants.
\newblock {\em SIAM J. Sci. Comput.}, 34(2):A737--A762, 2012.

\bibitem{Fornberg2011}
B.~Fornberg, E.~Larsson, and N.~Flyer.
\newblock Stable computations with {G}aussian {R}adial {B}asis {F}unctions.
\newblock {\em SIAM J. Sci. Comput.}, 33(2):869--892, 2011.

\bibitem{Kanagawa2016b}
M.~Kanagawa, B.~K. Sriperumbudur, and K.~Fukumizu.
\newblock Convergence guarantees for kernel-based quadrature rules in
  misspecified settings.
\newblock {\em Advances in Neural Information Processing Systems}, 29, 2016.

\bibitem{Kanagawa2019}
M.~Kanagawa, B.~K. Sriperumbudur, and K.~Fukumizu.
\newblock Convergence analysis of deterministic kernel-based quadrature rules
  in misspecified settings.
\newblock {\em Foundations of Computational Mathematics}, Jan 2019.

\bibitem{Karvonen2020a}
T.~Karvonen, G.~Wynne, F.~Tronarp, C.~Oates, and S.~S{\"a}rkk{\"a}.
\newblock Maximum likelihood estimation and uncertainty quantification for
  gaussian process approximation of deterministic functions.
\newblock {\em SIAM/ASA Journal on Uncertainty Quantification}, 8(3):926--958,
  2020.

\bibitem{Liang2020}
T.~Liang and A.~Rakhlin.
\newblock Just interpolate: Kernel ``ridgeless'' regression can generalize.
\newblock {\em The Annals of Statistics}, 48(3):1329--1347, 2020.

\bibitem{McCourt2017}
M.~McCourt and G.~E. Fasshauer.
\newblock Stable likelihood computation for {G}aussian random fields.
\newblock In {\em Recent Applications of Harmonic Analysis to Function Spaces,
  Differential Equations, and Data Science}, pages 917--943. Springer, 2017.

\bibitem{Narcowich2004}
F.~J. Narcowich, J.~D. Ward, and H.~Wendland.
\newblock Sobolev bounds on functions with scattered zeros, with applications
  to {R}adial {B}asis {F}unction surface fitting.
\newblock {\em Mathematics of Computation}, 74(250):743--763, 2005.

\bibitem{Narcowich2006}
F.~J. Narcowich, J.~D. Ward, and H.~Wendland.
\newblock Sobolev error estimates and a {B}ernstein inequality for scattered
  data interpolation via {R}adial {B}asis {F}unctions.
\newblock {\em Constructive Approximation}, 24(2):175--186, Sep 2006.

\bibitem{Pagliana2020}
N.~Pagliana, A.~Rudi, E.~De~Vito, and L.~Rosasco.
\newblock Interpolation and learning with scale dependent kernels.
\newblock {\em arXiv preprint arXiv:2006.09984}, 2020.

\bibitem{Pazouki2011}
M.~Pazouki and R.~Schaback.
\newblock Bases for kernel-based spaces.
\newblock {\em J. Comput. Appl. Math.}, 236(4):575--588, 2011.

\bibitem{scikit-learn}
F.~Pedregosa, G.~Varoquaux, A.~Gramfort, V.~Michel, B.~Thirion, O.~Grisel,
  M.~Blondel, P.~Prettenhofer, R.~Weiss, V.~Dubourg, J.~Vanderplas, A.~Passos,
  D.~Cournapeau, M.~Brucher, M.~Perrot, and E.~Duchesnay.
\newblock Scikit-learn: Machine learning in {P}ython.
\newblock {\em Journal of Machine Learning Research}, 12:2825--2830, 2011.

\bibitem{Rasmussen2006}
C.~E. Rasmussen and C.~K.~I. Williams.
\newblock {\em Gaussian Processes for Machine Learning}.
\newblock The MIT Press, 2006.

\bibitem{Richards2021}
D.~Richards, J.~Mourtada, and L.~Rosasco.
\newblock Asymptotics of ridge (less) regression under general source
  condition.
\newblock In {\em International Conference on Artificial Intelligence and
  Statistics}, pages 3889--3897. PMLR, 2021.

\bibitem{Saitoh2016}
S.~Saitoh and Y.~Sawano.
\newblock {\em Theory of Reproducing Kernels and Applications}.
\newblock Developments in Mathematics; 44. Springer, Singapore, 2016.

\bibitem{Scheuerer2013}
M.~Scheuerer, R.~Schaback, and M.~Schlather.
\newblock Interpolation of spatial data -- a stochastic or a deterministic
  problem?
\newblock {\em European Journal of Applied Mathematics}, 24(4):601--629, 2013.

\bibitem{Schoelkopf2002}
B.~Sch{\"o}lkopf and A.~Smola.
\newblock {\em Learning with Kernels}.
\newblock The MIT Press, 2002.

\bibitem{Shawe-Taylor2004}
J.~Shawe-Taylor and N.~Cristianini.
\newblock {\em Kernel Methods for Pattern Analysis}.
\newblock Cambridge University Press, 2004.

\bibitem{Steinwart2008}
I.~Steinwart and A.~Christmann.
\newblock {\em Support Vector Machines}.
\newblock Science + Business Media. Springer, 2008.

\bibitem{Matlab}
{T}he~{M}ath{W}orks {I}nc.
\newblock {MATLAB} {R2021}b {S}tatistics and {M}achine {L}earning {T}oolbox,
  {N}atick, {M}assachusetts, {USA}.

\bibitem{Wendland2005}
H.~Wendland.
\newblock {\em Scattered {D}ata {A}pproximation}, volume~17 of {\em Cambridge
  Monographs on Applied and Computational Mathematics}.
\newblock Cambridge University Press, Cambridge, 2005.

\bibitem{ZhZh2013b}
H.~Zhang and L.~Zhao.
\newblock On the inclusion relation of {R}eproducing {K}ernel {H}ilbert
  {S}paces.
\newblock {\em Analysis and Applications}, 11(02):1350014, 2013.

\bibitem{Zwicknagl2009}
B.~Zwicknagl.
\newblock Power series kernels.
\newblock {\em Constructive Approximation}, 29(1):61--84, Feb 2009.

\end{thebibliography}
\bibliographystyle{abbrv}

\end{document}